\documentclass[a4paper, reqno]{amsart}

\usepackage[utf8]{inputenc}
\usepackage[T1]{fontenc}
\usepackage{amsmath, amssymb, amscd, amsthm,epsfig,epic,eepic}
\usepackage[english]{babel}
\usepackage{mathrsfs}
\usepackage{tikz-cd}
\usepackage{verbatim}
\usepackage{color}
\usepackage[all]{xy}
\usepackage[justify]{ragged2e}
\usepackage{dynkin-diagrams}

\usepackage[shortlabels]{enumitem}
\setlist[enumerate]{label=\rm{(\arabic*)}}
\setlist[enumerate,2]{label=\rm{(\roman*)}}
\setlist[itemize]{label=\raisebox{0.25ex}{\tiny$\bullet$}}
\renewcommand{\thefootnote}{\fnsymbol{footnote}}

\setlist[itemize]{leftmargin=*}
\setlist[enumerate]{leftmargin=*}
 
\usepackage[left=3cm,right=3cm,top=3cm,bottom=3cm]{geometry}

\makeatletter
     \let\oldfootnote\footnote
     \def\footnote{\@ifstar\footnote@star\footnote@nostar}
     \def\footnote@star#1{{\let\thefootnote\relax\footnotetext{#1}}}
     \def\footnote@nostar{\oldfootnote}
     \makeatother

\newcommand{\K}{\mathbb{K}}
\newcommand{\N}{\mathbf{N}}
\newcommand{\Z}{\mathbf{Z}}

\newcommand{\T}{\mathbb{T}}
\newcommand{\red}{\mathrm{red}}
 
\makeatletter
\newcommand*{\defeq}{\mathrel{\rlap{%
                      \raisebox{0.3ex}{$\cdot$}}%
                      \raisebox{-0.3ex}{$\cdot$}}%
                      =}
\makeatother

\newcommand{\g}{\mathfrak{g}} 
\newcommand{\hhh}{\mathfrak{h}}

\newcommand{\Gm}{\mathbf{G}_m}
\newcommand{\Ga}{\mathbf{G}_a}

\DeclareMathOperator{\Supp}{Supp}

\DeclareMathOperator{\height}{ht}
\DeclareMathOperator{\Lie}{Lie}

\DeclareMathOperator{\SL}{SL}

\DeclareMathOperator{\Pic}{Pic}

\DeclareMathOperator{\ad}{ad}

\usepackage{xcolor, color}
\usepackage{hyperref, cleveref}
\hypersetup{
    colorlinks=true,
    linkcolor=teal,
    citecolor=teal,
    filecolor=teal,
    urlcolor=teal,
    linktocpage=true
}
 
\theoremstyle{plain}

\newtheorem{theorem}{Theorem}[section]
\newtheorem{lemma}[theorem]{Lemma}
\newtheorem{proposition}[theorem]{Proposition}
\newtheorem{corollary}[theorem]{Corollary}
\newtheorem{theorem*}{Theorem}
\newtheorem{proposition*}[theorem*]{Proposition}
\newtheorem{lemma*}[theorem*]{Lemma}
\newtheorem*{proposition**}{Proposition}
\newtheorem*{conventions*}{Conventions}
\newtheorem{comment*}{Comment}

\theoremstyle{definition}
\newtheorem{definition}[theorem]{Definition}
\newtheorem{remark}[theorem]{Remark}

\newtheorem{example}[theorem]{Example}
\numberwithin{equation}{section}
 
﻿

\begin{document}

\title{Horospherical subgroups in positive characteristic}

\author{Matilde Maccan}
    \address{Fakultät für Mathematik, RUB, Universitätsstr. 150, 44801 Bochum, Germany}
	\email{matilde.maccan@ruhr-uni-bochum.de} 

\author{Ronan Terpereau}
\address{Universit\'e de Lille, CNRS, UMR 8524 -- Laboratoire Paul Painlev\'e, F-59000 Lille, France}
\email{ronan.terpereau@univ-lille.fr}

\date{\today}

\begin{abstract}
We investigate horospherical homogeneous spaces---a class of spherical homogeneous spaces encompassing both flag varieties and algebraic tori---over {algebraically closed} fields of characteristic $p>0$, and establish their complete classification for $p > 2$.
\end{abstract}



\maketitle
\setcounter{tocdepth}{1}
\tableofcontents

\section{Introduction}
The main goal of this article is to initiate the study of horospherical varieties over {algebraically closed} fields of positive characteristic. 
Horospherical varieties form a natural subclass of spherical varieties, encompassing both flag varieties and toric varieties, yet their combinatorial structure is considerably more tractable. The combinatorial classification proceeds in two stages: first, one classifies the underlying homogeneous spaces (the focus of the present article); next, one applies Luna–Vust theory to classify equivariant embeddings of a fixed homogeneous space—a program already carried out by Knop~\cite{Kno91} over an arbitrary algebraically closed base field. 

Building on this, our ultimate aim (in a forthcoming paper) is to classify smooth projective horospherical varieties with Picard number one, in analogy with Pasquier's work~\cite{Pas09} over {algebraically closed} fields of characteristic zero, which produced new families of almost homogeneous varieties with rich and intricate geometry (see, e.g., \cite{Kim17,KP18,GPPS22,PM23}).

\smallskip

We now turn to the main definitions and setup.
Let $G$ be a (connected) reductive group over an algebraically closed field of characteristic $p > 0$. 
If $H$ is a subgroup (scheme) of $G$, we denote by $N_G(H)$ the normalizer of $H$ in $G$.
A subgroup $H \subset G$ is called \emph{horospherical} if it contains a maximal unipotent smooth connected subgroup of $G$. 
A subgroup $H$ is called \emph{strongly horospherical} if, in addition, $N_G(H)$ is a parabolic subgroup of $G$.

Note that if we were in characteristic $0$, the notions of horospherical subgroups and strongly horospherical subgroups would align, as the normalizer of any horospherical subgroup is then necessarily parabolic (see \cite[Proposition 2.2]{Pas08}). This is not necessarily the case in positive characteristic as the next example shows.

\begin{example}
    \label{Binfinity}
    Let $p=2$ and consider the following connected subgroup of $\SL_2$, whose definition can be found in \cite[\S~2]{Knop}:
    \[
    H=\mathcal{B}(\infty):= \left\{
    \begin{bmatrix}
        a & b \\ a+a^3 & a^3+ (1+a^2)b
    \end{bmatrix} \ \middle| \  a^4 = 1
    \right\} \subset \SL_2.
    \]
Setting $a=1$, we see that $H$ indeed contains $U=H_{\red}$; however, an explicit computation done in loc.~cit.~shows that the normalizer of $H$ is equal to $H$ itself. In particular, $H$ is horospherical but not strongly horospherical. Note also that $\Lie(H)=\Lie(B)$ is $T$-stable (for the restriction of the adjoint action), where $B \subset \SL_2$ is a Borel subgroup and $T \subset B$ is a maximal torus, although $H$ itself is not normalized by $T$.
\end{example}

However, and quite surprisingly, it turns out that the existence of horospherical but not strongly horospherical subgroups occurs only in characteristic 
$2$. The main result of this article is as follows:

\begin{theorem}[{Theorem~\ref{th: general case}}]
Assume that $p \geq 3$. Let $G$ be a reductive group and let $H$ be a subgroup of $G$.
Then $H$ is horospherical if and only if $H$ is strongly horospherical.
\end{theorem}

Moreover, when $p \geq 3$, any strongly horospherical subgroup $H \subset G$ satisfies that $N_G(H)/H$ is a torus (see Proposition~\ref{prop P/H is a torus}), whereas this is not always the case when $p=2$ (see Example~\ref{ex: not always a torus when p=2}). Consequently, over a field of characteristic $p \geq 3$, horospherical subgroups are easily classified:

\begin{corollary}[{Corollary~\ref{cor: classification strongly horo}}]
Assume that $p \geq 3$ {and denote by $\mathcal{P}$ the set of functions $\Phi^+ \to \mathbb{N} \cup \{\infty\}$ associated with parabolic subgroups of $G$ (see Section~\ref{sec: associated function}).} 
There is a one-to-one correspondence between the set of conjugacy classes of horospherical subgroups of $G$ and the set of pairs $(\varphi, M)$, where $\varphi \in \mathcal{P}$ and $M$ is a sublattice of $X^\ast(P_\varphi)$, with $P_\varphi$ the standard parabolic subgroup of $G$ associated with the function $\varphi$.
\end{corollary}

\begin{remark}
It would be interesting to extend the classification of horospherical subgroups of $G$ to the case $p=2$. In this setting two additional difficulties occur: 
\begin{itemize}
\item There exist strongly horospherical subgroups $H \subset G$ such that $N_G(H)/H$ is not a torus (see Example~\ref{ex: not always a torus when p=2}). This issue can be circumvented by replacing $N_G(H)$ with the \emph{smaller} parabolic subgroup $HT$ (see Remark~\ref{rk: classification in the case p=2}).
\item More seriously, there exist horospherical subgroups of $G$ that are not strongly horospherical (see Example~\ref{Binfinity}). Hence restricting to strongly horospherical subgroups is not sufficient.
\end{itemize}
\end{remark}

Let us note that, even when $p \neq 2$, the situation over a field of characteristic $p>0$ is much richer than over a field of characteristic $0$. Consider, for instance, the toy case $G=\SL_2$. Over a field of characteristic $0$, any non-trivial horospherical subgroup of $\SL_2$ is conjugate either to a Borel subgroup $B \subset \SL_2$ or to $\boldsymbol{\mu}_n \ltimes U$, with $U := R_u(B)$ {and $\boldsymbol{\mu}_n$ the subgroup of the multiplicative group $\Gm$ defined by $t^n = 1$}. In contrast, over a field of characteristic $p \geq 3$, any non-trivial horospherical subgroup of $\SL_2$ is conjugate \emph{to the pullback, via an iterated Frobenius morphism}, of either $B$ or $\boldsymbol{\mu}_n \ltimes U$ for some $n \geq 1$ (see Section~\ref{sec: SL2 case}). 

\medskip

As mentioned earlier, now that a complete classification of horospherical homogeneous spaces has been achieved (for $p \geq 3$), we plan in forthcoming work to study smooth projective horospherical varieties of Picard number one, with the goal of identifying new and interesting families—beyond flag varieties—that do not occur in characteristic $0$.

\subsection{Outline of the article}
Section~\ref{sec: parabolic groups} reviews the classification of parabolic subgroup schemes for $p \geq 3$, together with a description of their character lattices. 
In Section~\ref{sec: Lie algebras sec}, we show that almost all Lie subalgebras of $\Lie(G)$ containing $\Lie(U)$ are $T$-stable when $p \geq 3$. 
Section~\ref{sec: strongly horo subgroups} is devoted to strongly horospherical subgroups of $G$: we establish some basic properties and then provide a complete classification in the case $p \geq 3$ (Corollary~\ref{cor: classification strongly horo}). 
Finally, Section~\ref{sec: main section} presents the proof of our main result (Theorem~\ref{th: general case}).

\subsection{Notation and setting}
We work over an algebraically closed field $\K$ of characteristic $p > 0$.  
Throughout, we consider group schemes of finite type over $\K$, so all subgroups are understood as subgroup \emph{schemes} and are not necessarily smooth. For background on the theory of group schemes, we primarily refer to \cite{DG} and \cite{Milne}, and to a lesser extent \cite{Jan}.

Let $G$ be a (smooth) connected reductive algebraic group.  
{Since any connected reductive algebraic group $G$ admits a finite central isogeny from a product $C \times G'$, where $C \subset G$ is a maximal central torus and $G' \subset G$ is the derived subgroup, we may assume from now on, without loss of generality, that $G \simeq C \times G'$, where $C$ is a torus and $G'$ is a simply-connected semisimple group.}

We fix once and for all a Borel subgroup $B \subset G$ and a maximal torus $T \subset B$. Let $U$ denote the unipotent radical of $B$. {A \emph{parabolic} subgroup is a subgroup $P \subset G$ containing a Borel subgroup of $G$; since we are working over an algebraically closed field, this is equivalent to the homogeneous space $G/P$ being projective. Parabolic subgroups are \emph{not} assumed to be smooth in general.} Given a subgroup $H \subset G$, we denote by $N_G(H)$ the (scheme-theoretic) normalizer of $H$ in $G$, {and by $H_{\mathrm{red}}$ its associated reduced subscheme, which is a subgroup of $H$.}

Let $\Phi \supset \Phi^+ \supset \Delta$ be the sets of roots, positive roots, and simple roots of $G$ with respect to $T$ and $B$, respectively. For each root $\gamma \in \Phi$, let $U_\gamma$ denote the corresponding root subgroup, $u_\gamma \colon \mathbb{G}_a \to U_\gamma$ the associated root homomorphism, and $\mathfrak{g}_\gamma$ the corresponding root subspace in $\g = \Lie(G)$. For a simple root $\alpha \in \Delta$, let $P^\alpha$ denote the maximal smooth parabolic subgroup not containing $U_{-\alpha}$.

Let $F^m \colon G \to G^{(m)}$ be the $q$-th (relative) Frobenius morphism, where $q = p^m$ for some $m \geq 1$. Its kernel is denoted by ${}_m G$. If $H$ is a subgroup of $G$, then $F^m$ induces a short exact sequence of $\K$-group schemes
\[
1 \to {}_m H \to H \to F^m(H) \to 1.
\]
Note that $F^m(H) \subset H^{(m)}$ with equality when $H$ is smooth (see e.g.~\cite[Part I, \S~9.5]{Jan}).
Given a subgroup $H$ of $G$, we call \emph{height} of $H$ the smallest integer $m$ such that $F^m(H) \simeq H/{}_m H$ is smooth. In particular, a connected finite group scheme $J$ has height $m$ if $m$ is the smallest integer such that $J={}_mJ$. 
For any integer $n \geq 1$, we denote by $\boldsymbol{\mu}_n$ the subgroup of the multiplicative group $\Gm$ defined by $t^n = 1$; in particular, $\boldsymbol{\mu}_{p^r} = {}_r \Gm$. We denote by $\boldsymbol{\alpha}_{p^r}$ the $r$-th Frobenius kernel of the additive group $\Ga$. {A group scheme is called \emph{infinitesimal} if its associated reduced subscheme consists of a single point; this is equivalent to the group scheme having finite height.}

{For the notion of a \emph{restricted Lie algebra} (or \emph{$p$-Lie algebra}), which will be used throughout this article, we refer to \cite[Definition 10.39]{Milne}.
Recall that a restricted Lie algebra is a Lie algebra $\g$  endowed with a map $\g \to \g,\ x \mapsto x^{[p]}$, satisfying certain compatibility conditions. In particular, whenever a Lie algebra arises from an algebraic group (in characteristic $p$), it carries a natural restricted structure.}

\section{Parabolic subgroups and classification for \texorpdfstring{$p \geq 3$}{p>2}}\label{sec: parabolic groups}
Over an algebraically closed field of characteristic $0$, it is well known that any horospherical subgroup is normalized by a parabolic subgroup, and that this parabolic subgroup is uniquely determined, up to conjugation, by a subset of the simple roots of $G$. In positive characteristic, the situation becomes more intricate; nonetheless, parabolic subgroups remain crucial in the study of horospherical subgroups. To address this, we begin by presenting the classification of parabolic subgroups for $p \geq 3$, which synthesizes several classical results. Specifically, this classification builds upon the works of \cite{Wenzel,HL93} for the case $p \geq 5$ and of \cite{Maccan,Maccan2} to cover the remaining cases when $p=3$.

\subsection{Associated function} \label{sec: associated function}
Let $G$ {be a simply-connected semisimple group} over $\K$.
For a (non-necessarily smooth) parabolic subgroup $P$ with reduced part $P_{\rm{red}}$, we denote as
\begin{align}
\label{infinitesimal}
U_P^- \defeq P \cap R_u(P_{\rm{red}}^-)
\end{align}
its intersection with the unipotent radical of the opposite reduced parabolic of $P_{\rm{red}}$. The subgroup $U_P^-$ is unipotent, infinitesimal and satisfies
\begin{align}
\label{product}
U_P^- = \prod_{\gamma \in \Phi^+ \backslash \Phi_I} (U_P^- \cap U_{-\gamma}) \quad \text{and} \quad P = U_P^- \times P_{\rm{red}},
\end{align}
where both identities are isomorphisms of schemes given by the multiplication of $G$; {see \cite[Lemma~2 and Proposition~4]{Wenzel}}. Thus, $P$ can be recovered from its reduced part $P_{\rm{red}}$, together with its intersections with all the root subgroups contained in the opposite unipotent radical $R_u(P_{\rm{red}}^-)$. Let us reformulate this statement in a more combinatorial fashion, introducing a numerical function. {This follows \cite[Notation 9]{Wenzel}.} Recall that we denote the kernel of the $n$-th iterated Frobenius of the additive group $\Ga$ as $\boldsymbol{\alpha}_{p^n}$; while  $\boldsymbol{\alpha}_{p^\infty}$ is understood to be $\Ga$.

\begin{definition}
\label{def_varphi}
    Let $P$ be a parabolic subgroup of $G$ containing $B$. The \emph{associated function} 
    \[
    \varphi \colon \Phi^+ \longrightarrow \N \cup \{\infty\}
    \]
    is defined by the formula
    \[
 \forall \gamma \in \Phi^+,\ \ P \cap U_{-\gamma}= u_{-\gamma} ({\boldsymbol{\alpha}}_{p^{\varphi(\gamma)}}).
\]
\end{definition}

In other words, any positive root $\gamma$ not belonging to the root system of the Levi subgroup $P_{\mathrm{red}} \cap P_{\mathrm{red}}^-$ is sent to the natural number corresponding to the height of the finite unipotent subgroup $P \cap U_{-\gamma}$, while all other roots are sent to infinity. 

For instance, {in characteristic zero, or when $P$ is a smooth parabolic subgroup, the associated function $\varphi$ takes only the values $0$ and $\infty$. As a first example specific to positive characteristic, let us mention that} the function associated to the parabolic subgroup $_mG P^\alpha$ sends all positive roots to $\infty$, except for those whose support contains $\alpha$, which are assigned the value $m$.

\begin{theorem}[{\cite[Theorem 10]{Wenzel}}]
\label{thm:numericalfunction}
The parabolic subgroup $P$ is uniquely determined by the function $\varphi$, with no assumption on the characteristic or on the Dynkin diagram of $G$.
\end{theorem}

The following is (probably) a well-known result for experts, but for which the authors could not locate any reference in the literature. In the case of a reduced parabolic subgroup, this is a result due to Chevalley; see \cite[Corollary 17.49]{Milne}.

\begin{corollary}
    \label{cor: normalizerthm}
    Let $P$ be a parabolic subgroup of $G$; then $P = N_G(P)$.
\end{corollary}
\begin{proof}
    Let $Q = N_G(P)$; then $Q$ is a parabolic subgroup containing $P$, and in particular it is determined by the height of the intersections $Q \cap U_{-\gamma}$, where $\gamma$ ranges over the positive roots of $G$. To prove that $Q=P$, by \Cref{thm:numericalfunction} it is enough to show that the corresponding heights coincide. Assume that $P\cap U_{-\gamma}$ has height $r \in \N  \cup \{\infty\}$ and consider the copy of $\SL_2$ generated by $U_{-\gamma}$ and $U_\gamma$. Then assume (on the functor of points) that $c^{p^r}=0$ and compute
    \[
    \begin{bmatrix}
        x & y \\ w & z
    \end{bmatrix}
    \begin{bmatrix}
        a & b \\ c & d 
    \end{bmatrix}
    \begin{bmatrix}
        z & -y \\ -w & x
    \end{bmatrix} = \begin{bmatrix}
        a^\prime & b^\prime \\ c^\prime & d^\prime
    \end{bmatrix} = 
    \begin{bmatrix}
        a^\prime & b^\prime\\ waz+z^2c-w^2b-zdw & d^\prime
    \end{bmatrix}
    \]
    If we want the matrix on the right hand side to be still in $P$, we need
    \[
    0 = (c^\prime)^{p^r} = (wz(a-d) - w^2b)^{p^r}.
    \]
    Since there are no particular conditions on the variable $a,b,d$ (other than the fact that the corresponding matrix above must belong to $\SL_2$) we get that $w^{p^r} = 0$, i.e that the subgroup $Q \cap U_{-\gamma}$ has height $r$, as wanted.
    \end{proof}

\subsection{A very special isogeny} \label{sec: very special isogeny}
Assume that $G$ is simple and that its Dynkin diagram contains a single edge of multiplicity $p$: that is, either $G$ is of type $B_n$, $C_n$, or $F_4$ in characteristic $p = 2$, or $G$ is of type $G_2$ in characteristic $p = 3$. Under these assumptions, one can define the so-called \emph{very special isogeny} of $G$ as the quotient
\[
\pi_G \colon G \longrightarrow \overline{G}
\]
by a subgroup $K = K_G$ which is normal, non-central, {contained in the first Frobenius kernel}, and minimal with respect to these properties.

This subgroup is uniquely determined by its Lie algebra, which is defined as the smallest restricted Lie subalgebra containing the root subspaces corresponding to all \emph{short} roots. More precisely, the isogeny $\pi_G$ acts as Frobenius on the additive groups associated with short roots, and as the identity on those associated with long roots.

The quotient group $\overline{G}$ is simple and simply-connected, and its root system is dual to that of $G$. Moreover, the composition $\pi_{\overline{G}} \circ \pi_G$ equals the Frobenius morphism. For further details, see the original work of Borel and Tits \cite{BorelTits}, as well as \cite[Chapter~7]{CGP15}.

\subsection{Classification for \texorpdfstring{$p \geq 3$}{p>2}}
The following two results (Theorems~\ref{thm classification parabolic subgroups} and~\ref{thm classification parabolics 2}) provide a complete classification of the parabolic subgroups of a simply-connected semisimple group $G$ over an algebraically closed field of characteristic $p \geq 3$. {We will use this classification in the proof of the main result of this article (Theorem~\ref{th: general case}).}

Note that every parabolic subgroup of a product $G = G' \times G''$, where $G'$ and $G''$ are simply-connected simple groups, is the product of a parabolic subgroup of $G'$ and a parabolic subgroup of $G''$; {see \cite[Lemma~1.2]{Maccan2}, whose proof relies on the structure results of \Cref{thm:numericalfunction} and \cite[Theorem~8]{Wenzel}.}
This observation allows us to reduce to the case where $G$ is a simply-connected \emph{simple} group.

The first result concerns parabolic subgroups whose reduced part is maximal among reduced parabolic subgroups, and it was established in \cite{Maccan}.

\begin{theorem}\label{thm classification parabolic subgroups}
Assume that $p \geq 3$.  
Let $G$ be a simply-connected simple group, and let $P$ be a parabolic subgroup of $G$ such that $P_{\mathrm{red}} = P^\alpha$.  

If either $p \geq 5$, or $G$ is simply laced, or $p=3$ and $G$ is of type $B_n$, $C_n$, or $F_4$, then there exists an integer $r \geq 0$ such that  
\[
    P = {}_r G P^\alpha.
\]
Otherwise, $G$ is of type $G_2$ and $p=3$. In this case, there exists an integer $r \geq 0$ such that  
\[
    P = {}_r G P^\alpha 
    \quad \text{or} \quad 
    P = (F^r)^{-1}\bigl(K P^\alpha\bigr),
\]
where $K$ denotes the kernel of the very special isogeny 
$\pi_G \colon G \to \overline{G}$.
\end{theorem}

The second result completes the classification of parabolic subgroups of $G$ for $p \geq 3$ and is proven in \cite{Maccan2}.

\begin{theorem}
\label{thm classification parabolics 2}
    Let $P$ be a parabolic subgroup of a reductive group $G$, and let $\beta_1,\ldots, \beta_r$ be the simple roots of $G$ such that $P_{\red}$ is the intersection of the parabolic subgroups $P^{\beta_i}$. 
    Then
    \[
    P = \bigcap_{i=1}^r Q^i,
    \] where $Q^i$ is the smallest subgroup of $G$ containing both $P$ and $P^{\beta_i}$. In particular, every parabolic subgroup $P$ can be expressed as the intersection of parabolic subgroups whose reduced part is maximal.
\end{theorem}

\subsection{Character lattice} \label{sec character lattice of parab subgroups}
We now describe explicitly the character lattice {of an arbitrary parabolic subgroup of a simply-connected simple group. The reductive case follows by taking direct sums.}
Such a description will be needed later, to show that any horospherical subgroup is strongly horospherical when $p \geq 3$. 
In characteristic at least $5$, the following result, due to \cite[Section~2]{HL93}, holds.

\begin{lemma}
	\label{characters geq5}
	Let $p \geq 5$. Then the character lattice of the parabolic subgroup
	\[
	P = \bigcap_\alpha ({}_{r_\alpha} G P^\alpha)
	\]
	is given by
	\[
	X^*(P) = \bigoplus_\alpha \mathbf{Z} p^{r_\alpha} \varpi_\alpha,
	\]
	where $\varpi_\alpha$ denotes the fundamental weight associated with the simple root $\alpha$.
\end{lemma}

For $p=3$, the only additional input in the classification of parabolic subgroups 
comes from the very special isogeny, which occurs when the group {is of type} $G_2$. 
In that case, we denote by
\[
{}^rK := (F^r)^{-1}K
\]
the preimage of $K$ under the $r$-th iterated Frobenius. 
By \Cref{thm classification parabolic subgroups} and \Cref{thm classification parabolics 2}, 
an arbitrary parabolic subgroup is then of the form 
\begin{align}
	\label{parabolic char 3}
	P = \left( \bigcap_{{\alpha\in I}} {}_{r_\alpha} G P^\alpha \right) \cap \left( \bigcap_{{\nu \in J}} {}^{r_\nu-1}K P^\nu \right),
\end{align}
for some integers $r_\alpha \geq 0$ and $r_\nu \geq 1$ {and some distinct subsets of simple roots $I$ and $J$.}
This description is unique under the assumption that each factor in the above intersection is minimal with respect to inclusion.

\begin{lemma}
	\label{characters char 3}
	Let $p=3$, and let $P$ be a parabolic subgroup as in \eqref{parabolic char 3}. 
	Then its character group is given by
	\[
	X^*(P) = \bigoplus_{{\beta \in I \cup J}} \mathbf{Z} p^{r_\beta}\varpi_\beta.
	\] 
\end{lemma}

\begin{proof}
	Recall that the character group of $P$ is naturally identified with the Picard group of the corresponding homogeneous space $G/P$; see for example \cite[Remark 1.4.2]{BrionLectures}.
	From the description of the Picard group in term of Schubert divisors (given in \cite[Theorem~3.12]{Maccan}) together with the notations of \Cref{thm classification parabolics 2}, the character group of $P$ is the direct sum of the character groups of the corresponding $Q^i$. 
	More precisely, when $P$ is of the form \eqref{parabolic char 3}, we have
	\begin{align*}
		X^\ast(P)  
		= \Pic(G/P) 
		&= \left( \bigoplus_\alpha \Pic(G/({}_{r_\alpha}GP^\alpha)) \right) 
		  \oplus \left( \bigoplus_\nu \Pic(G/({}^{r_\nu-1}KP^\nu))\right) \\
		&= \left( \bigoplus_\alpha \Z p^{r_\alpha} \varpi_\alpha \right) 
		  \oplus \left( \bigoplus_\nu X^\ast ({}^{r_\nu-1}KP^\nu)\right).
	\end{align*}
	Since pulling back by a Frobenius multiplies the generator of the character group by $p$, it only remains to compute the character group of $KP^\nu$ for a simple root $\nu$. 
	We claim that
	\[
	X^*(KP^\nu) = \Z p \varpi_\nu,
	\]
	from which it follows that the character group of ${}^{r_\nu -1}KP^\nu$ is freely generated by $p^{r_\nu}\varpi_\nu$, for any $r_\nu \geq 1$.

	Indeed, since $P^\nu \subset KP^\nu \subset {}_1G P^\nu$, we obtain the inclusions
	\[
	\Z \varpi_\nu = X^\ast (P^\nu) 
	\;\subset\; X^\ast(KP^\nu) 
	\;\subset\; X^\ast ({}_1 G P^\nu) = \Z p\varpi_\nu.
	\]
	Because there are no intermediate sublattices, one of the two inclusions must be strict. 
	Suppose, for contradiction, that $\varpi_\nu$ is defined on $KP^\nu$. 
	Then it is also defined on the restriction to $KP^\nu \cap G_\gamma$, where $\gamma$ is a short root containing $\nu$ in its support with multiplicity one (for instance, the sum of all simple roots). 
	This intersection is
	\[
	P^\prime := KP^\nu \cap G_\gamma 
	= {}_1G_\gamma (B \cap G_\gamma) 
	\simeq \left\{\begin{bmatrix}
		a & b \\ c & d
	\end{bmatrix} \colon c^p=0 \right\} \subset G_\gamma \simeq \SL_2.
	\]
	Since the multiplicity of $\nu$ in $\gamma$ is $1$, we have $\langle \varpi_\nu, \gamma^\vee \rangle = 1$. 
	Thus the restriction of $\varpi_\nu$ to $T \cap G_\gamma$ coincides with the fundamental weight $\omega$ of $G_\gamma$. 
	However, $X^*(P^\prime)$ is generated by $p\omega$, and $\omega$ itself is not defined on this parabolic. 
	Hence the restriction of $\varpi_\nu$ cannot be defined either, a contradiction. 
	We conclude that $X^*(KP^\nu)$ is freely generated by $p\varpi_\nu$.
\end{proof}

\section{Most Lie subalgebras containing \texorpdfstring{$\Lie(U)$}{Lie(U)} are \texorpdfstring{$T$}{T}-stable for \texorpdfstring{$p \geq 3$}{p>2}} \label{sec: Lie algebras sec}

Let $G$ be a reductive group, which we may assume, without loss of generality, to be of the form $G = C \times G'$, where $C$ is a torus and $G' = \prod_{j \in J} G_{(j)}$ is a simply connected semisimple group.
Assume that $p \geq 3$. In this section we prove that almost all Lie subalgebras of $\Lie(G)$ containing $\Lie(U)$ are $T$-stable (see Propositions~\ref{prop: hhh containing LieU} and~\ref{prop: LieH G2}	 for precise statements). This will be useful in the proof of Proposition~\ref{prop:LieH-T-stable}.

\smallskip

Throughout this section, we denote by
\[
\{X_\gamma^i, V_\alpha^i \colon \gamma \in \Phi_i, \, \alpha \in \Delta_i \}
\]
a Chevalley basis of $\Lie(G_{(i)})$, where $G_{(i)}$ is a simple factor of $G$. 
If $G$ is simple, we use the same notation, omitting the index $i$ throughout.
 
We will need the following bracket relations between elements of such a basis (see \cite[\S~25.2]{Humphreys} for details):
\begin{enumerate}[(i)]
    \item $[V_\alpha, V_\beta] = 0$ for all $\alpha, \beta \in \Delta$;
    \item $[X_{-\alpha}, X_\alpha] = V_\alpha$ for all $\alpha \in \Delta$;
    \item $[V_\alpha, X_\gamma] = \gamma(V_\alpha) X_\gamma$;
    \item $[X_\gamma, X_\delta] = 0$ if $\gamma + \delta$ is not a root;
    \item\label{item: bracket v} $[X_\gamma, X_\delta] = \mathcal{N}_{\gamma,\delta} X_{\gamma+\delta}$ if $\gamma + \delta$ is a root, in which case $\gamma - (\mathcal{N}_{\gamma,\delta} - 1)\delta$ is also a root, while $\gamma - \mathcal{N}_{\gamma,\delta} \delta$ is not.  
\end{enumerate}
 Note that the structure constants $\mathcal{N}_{\gamma,\delta} \in \{\pm 1, \pm 2, \pm 3\}$ satisfy $\mathcal{N}_{\gamma,\delta} = -\mathcal{N}_{\delta,\gamma}$. Moreover, in the simply-laced case, we have $\mathcal{N}_{\gamma,\delta} \in \{\pm 1\}$. Note also that, by restricting to the $\mathrm{SL}_2$ copy generated by a fixed root $\gamma$, the above relations imply:
\begin{enumerate}[(vi)]
    \item $[[X_{-\gamma}, X_\gamma], X_\gamma] = 2X_\gamma$.
\end{enumerate}

\begin{remark}
	\label{rem: BnCn and p=3}
Let $G$ be a simple group over a field of characteristic $p \geq 3$, not of type $F_4$ or $G_2$. 
Fix a Lie subalgebra $\mathfrak{h} \subset \Lie(G)$. Then, for any positive root $\gamma \in \Phi^+$, if $X_{-\alpha} \in \mathfrak{h}$ for every $\alpha \in \operatorname{Supp}(\gamma)$, it follows that $X_{-\gamma} \in \mathfrak{h}$ as well. 
This is due to the fact that the relevant structure constants $\mathcal{N}_{\gamma,\delta}$ lie in $\{\pm 1, \pm 2\}$, and hence are nonzero in characteristic $p \geq 3$. 
For groups of type $F_4$ or $G_2$ over a field of characteristic $p \geq 5$, the same conclusion holds. {
Note, however, that these root systems have structure constants equal to $\pm 3$, which explains the restriction on the characteristic.}
\end{remark}

\begin{remark}
We follow the root system conventions of \cite{Bourbaki}.
\begin{itemize}
    \item For $F_4$: a basis is given by $\{\alpha_1, \alpha_2, \alpha_3, \alpha_4\}$, 
    where $\alpha_1$ and $\alpha_2$ are long roots, and $\alpha_3$, $\alpha_4$ are short roots.  
  \begin{center}
\dynkin[edge length=.95cm,
labels*={\alpha_1,\alpha_2,\alpha_3,\alpha_4}]F4
\end{center}
    \item For $G_2$: a basis is given by $\{\alpha_1, \alpha_2\}$, where $\alpha_1$ is the short root and $\alpha_2$ is the long root.
    \begin{center}
\dynkin[edge length=.95cm,
labels*={\alpha_1,\alpha_2}]G2
\end{center}	
\end{itemize}
\end{remark}

\begin{lemma}\label{prop normalizersreduced F4}
Let $G$ be a simple group of type $F_4$ over a field of characteristic $p \geq 3$. 
Let $\mathfrak{h}$ be a Lie subalgebra of $\Lie(G)$ containing $\Lie(U)$. 
Then, for any positive root $\gamma$, if $X_{-\alpha} \in \mathfrak{h}$ for every $\alpha \in \Supp(\gamma)$, it follows that $X_{-\gamma} \in \mathfrak{h}$.
\end{lemma}

\begin{proof}
If $\gamma$ is simple, there is nothing to prove. If $\Supp(\gamma) \subsetneq \Delta$, then we may restrict to a subsystem of type $B_3$ (generated by $\alpha_1, \alpha_2, \alpha_3$) or $C_3$ (generated by $\alpha_2, \alpha_3, \alpha_4$), where no structure constant is equal to $\pm 3$, so the claim follows from Remark~\ref{rem: BnCn and p=3} in these cases.

Thus, it remains to consider the case where $\Supp(\gamma) = \Delta$. We will show that if $X_{-\alpha_i} \in \mathfrak{h}$ for all $i = 1, \dots, 4$, then all negative root vectors belong to $\mathfrak{h}$. From the previous step (restriction to a subsytem of type $B_3$), we already know that $X_{-\alpha_1 - \alpha_2 - \alpha_3} \in \mathfrak{h}$, which serves as a starting point. There are exactly 10 positive roots to check, and we proceed inductively via brackets:

\begin{align*}
\gamma_1 &\coloneqq \alpha_1 + \alpha_2 + \alpha_3 + \alpha_4 & \Rightarrow && X_{-\gamma_1} \in \mathfrak{h} &  \quad\text{since} \quad \mathcal{N}_{\alpha_1 + \alpha_2 + \alpha_3,\, \alpha_4} = \pm 1, \\
\gamma_2 &\coloneqq \alpha_1 + \alpha_2 + 2\alpha_3 + \alpha_4 & \Rightarrow && X_{-\gamma_2} \in \mathfrak{h} &  \quad\text{since} \quad \mathcal{N}_{\gamma_1,\, \alpha_3} = \pm 1, \\
\gamma_3 &\coloneqq \alpha_1 + 2\alpha_2 + 2\alpha_3 + \alpha_4 & \Rightarrow && X_{-\gamma_3} \in \mathfrak{h} &  \quad\text{since}  \quad \mathcal{N}_{\gamma_2,\, \alpha_2} = \pm 1, \\
\gamma_4 &\coloneqq \alpha_1 + \alpha_2 + 2\alpha_3 + 2\alpha_4 & \Rightarrow && X_{-\gamma_4} \in \mathfrak{h} &  \quad\text{since}  \quad \mathcal{N}_{\gamma_2,\, \alpha_4} = \pm 2, \\
\gamma_5 &\coloneqq \alpha_1 + 2\alpha_2 + 3\alpha_3 + \alpha_4 & \Rightarrow && X_{-\gamma_5} \in \mathfrak{h} &  \quad\text{since}  \quad  \mathcal{N}_{\gamma_3,\, \alpha_3} = \pm 1, \\
\gamma_6 &\coloneqq \alpha_1 + 2\alpha_2 + 2\alpha_3 + 2\alpha_4 & \Rightarrow && X_{-\gamma_6} \in \mathfrak{h} &  \quad\text{since}  \quad \mathcal{N}_{\gamma_4,\, \alpha_2} = \pm 1, \\
\gamma_7 &\coloneqq \alpha_1 + 2\alpha_2 + 3\alpha_3 + 2\alpha_4 & \Rightarrow && X_{-\gamma_7} \in \mathfrak{h} &  \quad\text{since}  \quad \mathcal{N}_{\gamma_6,\, \alpha_3} = \pm 1, \\
\gamma_8 &\coloneqq \alpha_1 + 2\alpha_2 + 4\alpha_3 + 2\alpha_4 & \Rightarrow && X_{-\gamma_8} \in \mathfrak{h} &  \quad\text{since}   \quad \mathcal{N}_{\gamma_7,\, \alpha_3} = \pm 2, \\
\gamma_9 &\coloneqq \alpha_1 + 3\alpha_2 + 4\alpha_3 + 2\alpha_4 & \Rightarrow && X_{-\gamma_9} \in \mathfrak{h} &  \quad\text{since}  \quad \mathcal{N}_{\gamma_8,\, \alpha_2} = \pm 1, \\
\gamma_{10} &\coloneqq 2\alpha_1 + 3\alpha_2 + 4\alpha_3 + 2\alpha_4 & \Rightarrow && X_{-\gamma_{10}} \in \mathfrak{h} & \quad \text{since}  \quad \mathcal{N}_{\gamma_9,\, \alpha_1} = \pm 1.
\end{align*}
\end{proof}

\begin{remark}
	Assume that $p=3$ and that $G$ is a simple group of type $G_2$.
	Let $\mathfrak{h}$ be a Lie subalgebra of $\Lie(G)$ containing $\Lie(U)$. 
	Let $\gamma$ be a positive root of $G_2$. 
	Then it is not true that $X_{-\gamma}$ belongs to $\hhh$ if and only if $X_{-\alpha} \in \hhh$ for all $\alpha \in \Supp(\gamma)$. Actually, both implications fail.
	\begin{itemize}
		\item Recall that $K$ denotes the kernel of the very special isogeny (see Section~\ref{sec: very special isogeny}). 
		Consider the parabolic subgroup $H = K B$. Then, by \cite[Proposition 2.20]{Maccan}, we have
		\[
			\Lie(H) = \Lie(K) + \Lie(B) = \Lie(T) \oplus \Lie(U) \oplus 
			\K X_{-\alpha_1} \oplus 
			\K X_{-\alpha_1-\alpha_2} \oplus 
			\K X_{-2\alpha_1-\alpha_2}.
		\]
		In particular, $X_{-2\alpha_1-\alpha_2} \in \Lie(H)$, 
		but $X_{-\alpha_2} \notin \mathfrak{h}$. 
		Hence the implication 
		\[
			X_{-\gamma} \in \mathfrak{h} \;\;\Rightarrow\;\; 
			X_{-\alpha} \in \mathfrak{h}, \;\; \forall\, \alpha \in \Supp(\gamma)
		\]
		does not hold.
		
		\item Consider the parabolic subgroup $H' = K P^{\alpha_1}$: again by \cite[Proposition 2.20]{Maccan}, its Lie algebra is
		\[
			\Lie(H') = \Lie(H) \oplus \K X_{-\alpha_2}.
		\]
		Thus both $X_{-\alpha_1}, X_{-\alpha_2} \in \Lie(H')$. 
		However, the highest root $\gamma = 3\alpha_1 + 2\alpha_2$ does \emph{not} satisfy $X_{-\gamma} \in \Lie(H')$. 
		This shows that the converse implication
		\[
			\bigl( X_{-\alpha} \in \mathfrak{h}, \;\;\forall\, \alpha \in \Supp(\gamma) \bigr) 
			\;\;\Rightarrow\;\; X_{-\gamma} \in \mathfrak{h}
		\]
		is also false.
	\end{itemize}
\end{remark}

\begin{lemma}	\label{lem 3roots}
Let $G$ be a simple group of type $F_4$ over a field of {arbitrary characteristic}. 
For any positive root $\gamma \in \Phi^+$ which is not simple, and for any $i \in \{1,2,3,4\}$, the weights $3\gamma - 2\alpha_i$ and $3\gamma - \alpha_i$ are not roots. 
\end{lemma}

\begin{proof}
The two cases are analogous, so we only prove that the weight $3\gamma - 2\alpha_i$ is not a root.

	Let $l(\gamma) \geq 2$ denote the length of the positive root $\gamma$, and fix a simple root $\alpha_i$. If $\alpha_i \notin \operatorname{Supp}(\gamma)$, then there is nothing to prove, so assume $\alpha_i \in \operatorname{Supp}(\gamma)$.

	If $3\gamma - 2\alpha_i$ were a root, its length would be
	\[
	3\,l(\gamma) - 2 \in \{4,7,10\},
	\]
	because the highest root has length $11$. 	We list all positive roots $\delta$ of these lengths.

Those with $l(\delta) = 4$ are exactly:
	\[
	\alpha_1 + \alpha_2 + 2\alpha_3, \quad
	\alpha_2 + 2\alpha_3 + \alpha_4, \quad
	\alpha_1 + \alpha_2 + \alpha_3 + \alpha_4.
	\]
	For any choice of $i$, none of these roots $\delta$ satisfies that $\delta + 2\alpha_i$ has all coefficients divisible by $3$. Hence it is impossible that $\delta + 2\alpha_i = 3\gamma$.

	The same reasoning applies for the roots
	\[
	\alpha_1 + 2\alpha_2 + 2\alpha_3 + 2\alpha_4, \quad
	\alpha_1 + 2\alpha_2 + 3\alpha_3 + \alpha_4, \quad
	\alpha_1 + 3\alpha_2 + 4\alpha_3 + 2\alpha_4,
	\]
	which are respectively those of length $7$ or $10$.
	\end{proof}

\begin{lemma}
	\label{lem gamma minus alpha}
	For any (non-simple) positive root $\gamma$ of $G$, there is always some $\alpha \in \Supp(\gamma)$ such that $\gamma-\alpha$ is still a root.
	\end{lemma}
	
\begin{proof}
This fact is classical; see, for example, \cite[Corollary~1.12~(i)]{DF24}.
\end{proof}

\begin{lemma}	\label{F4: gamma + 2alpha no root}
Let $G$ be a simple group of type $F_4$ over a field {of characteristic $p \geq 3$}. 
Then for any (non-simple) positive root $\gamma$ there is some simple root
$\alpha \in \Supp(\gamma)$ such that  {the structure constant $\mathcal{N}_{\gamma,-\alpha}$ does not vanish.}
\end{lemma}

\begin{proof}
	As in the proof of Lemma~\ref{prop normalizersreduced F4}, for all the roots $\gamma$ such that the support of $\gamma$ is not the full basis of simple roots, then we can restrict to a subgroup of type $B_3$ or $C_3$, where none of the structure constants vanish {for $p\geq3$}. Thus, we only have to consider the roots $\gamma_1,\ldots, \gamma_{10}$ listed above, which are exactly those satisfying $\Supp(\gamma) = \Delta$. 
	
	Let us consider the following pairs $(\gamma_j,\alpha_{i(j)})$:
	\[
	(\gamma_1,\alpha_1), \,\, (\gamma_2,\alpha_3), \,\, (\gamma_3,\alpha_2), \,\, (\gamma_4, \alpha_4), \,\, (\gamma_5, \alpha_3), \,\, (\gamma_6, \alpha_2), \,\, {(\gamma_7, \alpha_4)}, \,\, (\gamma_8, \alpha_3), \,\, (\gamma_9, \alpha_2), \,\, (\gamma_{10}, \alpha_1). 
	\]
	Then, by looking at the full list of positive roots in type $F_4$, {we notice the following: for each $j= 1,\ldots,10$ the weight $\gamma_j-\alpha_{i(j)}$ is a root, while the weight $\gamma_j + \alpha_{i(j)}$ is not. By the bracket relation~\ref{item: bracket v}, this implies exactly that $\mathcal{N}_{\gamma_j, -\alpha_{i(j)}}$ is equal to $\pm 1$.
	}
\end{proof}

\begin{proposition} \label{prop: hhh containing LieU}
Assume that $p \geq 3$. Let $\mathfrak{h} \subset \Lie(G)$ be a Lie subalgebra containing $\Lie(U)$. If $p \geq 5$, or if $p = 3$ and $G$ does not contain any simple factor of type $G_2$, then $\mathfrak{h}$ is $T$-stable.
\end{proposition}

\begin{proof}  
Consider
\begin{align}
\label{Xgamma}
X = \sum_{j=1}^r X_j + W, \quad \text{where} \quad X_j = \sum_{\gamma \in \Phi_j^+} a_\gamma^j X_{-\gamma}^j \in \Lie(U_j^-),
\end{align}
for some $a_\gamma^j \in \K$ and $W \in \Lie(T)$. Fix a simple factor $G_{(i)}$, and define the support of $X$ with respect to this factor as
\[
\Supp(X, i) = \bigcup_{a_\gamma^i \neq 0} \Supp(\gamma) \subset \Delta_i.
\]

From now on, we drop the index $i$ and denote the Chevalley basis vectors for $G_{(i)}$ by $X_\delta$ and $V_\beta$.
Assume that either $p \geq 5$, or that $G_{(i)}$ is not of type $G_2$. We aim to show that $a_\gamma^i \neq 0$ implies $X_{-\gamma}^i \in \mathfrak{h}$.
In fact, we will prove a stronger statement: if $a_\gamma^i \neq 0$, then $X_{-\alpha}^i \in \mathfrak{h}$ for all $\alpha \in \Supp(X, i)$. This suffices, since either $G_{(i)}$ is of type $F_4$, in which case we may invoke \Cref{prop normalizersreduced F4}, or we appeal to \Cref{rem: BnCn and p=3}.

Let us prove this claim by induction on the \emph{length} of $X$ with respect to $G_{(i)}$:
\begin{align*}
l := l(X, i) = \max \left\{ l(\gamma) \colon \gamma \in \Phi_i^+, \, a_\gamma^i \neq 0 \right\},
\end{align*}
where $l(\gamma) = \sum_{\beta \in \Delta_i} n_\beta$ if $\gamma = \sum_{\beta} n_\beta \beta$.

\smallskip  

\noindent
\textbf{Base case:} $l = 1$. Then $X$ takes the form
\[
X = \sum_{j \neq i} X_j + \sum_{\alpha \in \Delta_i} a_\alpha X_{-\alpha} + W.
\]
Note that, for all $\beta \in \Delta$, the bracket $[X, X_\beta] \in \mathfrak{h}$ because $\Lie(U) \subset \mathfrak{h}$. 
Hence,
\[
V := \bigoplus_{\alpha \in \Supp(X,i)} \K V_\alpha \subset \mathfrak{h}.
\]
Now consider the linear map
\[
V \longrightarrow \bigoplus_{\alpha \in \Supp(X,i)} \K X_{-\alpha}, \quad V_\beta \mapsto [X, V_\beta],
\]
whose image lies in $\mathfrak{h}$. Its matrix is the Cartan matrix $(\alpha(V_\beta))_{\alpha,\beta \in \Supp(X,i)}$, which is invertible. Hence $X_{-\alpha} \in \mathfrak{h}$ for all $\alpha \in \Supp(X,i)$.

\smallskip 

\noindent
\textbf{Induction step:} Assume the claim holds for all vectors $Y$ with $l(Y) < l$. Let $X$ be as in \eqref{Xgamma}, with $l(X) = l \geq 2$. 
Pick $\alpha_0 \in \Supp(\gamma_0) \subset \Supp(X, i)$ such that $\gamma_0 \in \Phi_i^+$, $a_{\gamma_0} \neq 0$, and $l(\gamma_0) = l(X) = l$. Our goal is to show that $X_{-\alpha_0} \in \mathfrak{h}$ (then one easily concludes by induction).

\smallskip 

\noindent
\textbf{Case 1.} Suppose that $\delta_0 := \gamma_0 - \alpha_0 \in \Phi^+$. Then:
\[
[X, X_{\delta_0}] = a_{\gamma_0} \mathcal{N}_{-\gamma_0,\delta_0} X_{-\alpha_0} 
+ \sum_{\substack{\gamma - \delta_0 \in \Phi_i^+ \\ \gamma \neq \gamma_0}} a_{\gamma} \mathcal{N}_{-\gamma,\delta_0} X_{-\gamma + \delta_0}
+ W_0 + \text{(other terms in $\Lie(U)$)} \in \mathfrak{h},
\]
where $W_0 := a_{\delta_0}[X_{-\delta_0}, X_{\delta_0}] \in \Lie(T)$. 
Moreover, for any $\gamma$ such that $\gamma - \delta_0 \in \Phi_i^+$, we have $l(\gamma - \delta_0) < l(\gamma)$. In particular, we may replace the vector $X$ from \eqref{Xgamma} with the vector
\[
Y := a_{\gamma_0} \mathcal{N}_{-\gamma_0,\delta_0} X_{-\alpha_0} 
+ \sum_{\substack{\gamma - \delta_0 \in \Phi_i^+ \\ \gamma \neq \gamma_0}} a_\gamma \mathcal{N}_{-\gamma,\delta_0} X_{-\gamma + \delta_0} 
+ W_0 \in \mathfrak{h},
\]
which satisfies $l(Y) < l(X) = l$. 
By our assumption on $p$, or by \Cref{lem 3roots} (for $G_{(i)}$ of type $F_4$ and $p=3$), we have $\mathcal{N}_{-\gamma_0,\delta_0} \neq 0$. Hence, we conclude by induction that $X_{-\alpha_0} \in \mathfrak{h}$.

\smallskip  

\noindent
\textbf{Case 2.} If the above does not hold directly for $\alpha_0$, then there exists another simple root $\alpha_1 \in \Supp(\gamma_0)$ with $\alpha_1 \neq \alpha_0$, such that $\delta_1 := \gamma_0 - \alpha_1 \in \Phi^+$ and $\mathcal{N}_{\gamma_0, -\alpha_1} \neq 0$ (see Lemmas~\ref{lem gamma minus alpha} and~\ref{F4: gamma + 2alpha no root}).
Consider now:
\[
[X, X_{\alpha_1}] =  a_{\gamma_0} \mathcal{N}_{-\gamma_0, \alpha_1} X_{-\delta_1} 
+ \sum_{\substack{\gamma - \alpha_1 \in \Phi_i^+ \\ \gamma \neq \gamma_0,\, \alpha_1}} a_{\gamma} \mathcal{N}_{-\gamma,\alpha_1} X_{-\gamma + \alpha_1} 
+ a_{\alpha_1} V_{\alpha_1}
+ \text{(other terms in $\Lie(U)$)} \in \mathfrak{h}.
\]
As in the previous case, we replace the vector $X$ from \eqref{Xgamma} with the new vector
\[
Y' := a_{\gamma_0} \mathcal{N}_{-\gamma_0, \alpha_1} X_{-\delta_1} 
+ \sum_{\substack{\gamma - \alpha_1 \in \Phi_i^+ \\ \gamma \neq \gamma_0,\, \alpha_1}} a_{\gamma} \mathcal{N}_{-\gamma,\alpha_1} X_{-\gamma + \alpha_1} 
+ a_{\alpha_1} V_{\alpha_1} \in \mathfrak{h},
\]
which satisfies $l(Y') < l(X) = l$ and still has $\alpha_0$ in its support; indeed, $\alpha_0 \in \Supp(\delta_1) \subset \Supp(Y', i)$. 
By induction, we conclude that $X_{-\alpha_0} \in \mathfrak{h}$.
This concludes the proof of the proposition.
\end{proof}

\begin{example}	\label{ex: hb}  
Assume that $p=3$ and that $G$ is simple of type $G_2$. Recall that we denote as $K$ the kernel of the very special isogeny.
Let us denote as $\alpha_1$ the short simple root and as $\alpha_2$ the long one. 
Fix a nonzero scalar $\lambda \in \K^\times$ and define
	\[
	X_\lambda \defeq \lambda X_{-\alpha_2}+X_{-3\alpha_1-\alpha_2}.
	\]
	Consider the following linear subspace of $\Lie(G)$:
	\begin{align}
		\label{subalgebra_hb}
		\hhh_\lambda  \defeq \Lie (KB) \oplus \K X_\lambda =\Lie(U) \oplus \Lie(T) \oplus \K X_{-\alpha_1}\oplus \K X_{-\alpha_1-\alpha_2}\oplus \K X_{-2\alpha_1-\alpha_2} \oplus \K X_\lambda.
	\end{align}
	We will show that $\hhh_\lambda$ is a restricted Lie subalgebra of $\Lie(G)$ that is not $T$-stable.

	By inspecting the structure constants, we find that $\mathcal{N}_{-2\alpha_1-\alpha_2,-\alpha_1}$, $\mathcal{N}_{-\alpha_1-\alpha_2,\alpha_1}$, and similar terms vanish (recall that $p=3$). Furthermore, a direct computation shows that both $V_{\alpha_1}$ and $V_{\alpha_2}$ stabilize $X_\lambda$. This confirms that $\hhh_\lambda$ is closed under the Lie bracket, i.e., it is a Lie subalgebra of $\Lie(G)$. Also, it is clear from the definition of $\hhh_\lambda$ that this Lie subalgebra is not $T$-stable, since it does not contain the subspace $\K X_{-\alpha_2} \oplus \K X_{-3\alpha_1-\alpha_2}$.

	Next, we check that $\hhh_\lambda$ is stable under the $p$-mapping. Since $\Lie (KB)$ is already known to be a restricted Lie subalgebra, it suffices to compute $X_\lambda^{[p]} = X_\lambda^{[3]}$. By properties of the $p$-mapping, we have:
	\[
	\ad (X_\lambda^{[3]}) = \ad(X_\lambda)^3 = \ad (\lambda X_{-\alpha_2}+X_{-3\alpha_1-\alpha_2})^3 = 0,
	\]
	which implies that $X_\lambda^{[3]}$ lies in the center of $\Lie(G)$. Since this center is trivial, we conclude that $X_\lambda^{[3]} = 0$. Therefore, $\hhh_\lambda$ is indeed a restricted Lie subalgebra of $\Lie(G)$.
	
\end{example}

{
As mentioned earlier, our main goal in this section is to prove that almost all Lie subalgebras $\hhh$ of $\Lie(G)$ containing $\Lie(U)$ are $T$-stable. We have already treated the case where $p \ge 5$, or $p = 3$ and $G$ has no simple factor of type $G_2$, in Proposition~\ref{prop: hhh containing LieU}. It remains to consider the case where $p = 3$ and $G$ has at least one factor of type $G_2$. We will show that in this case either $\hhh$ is $T$-stable, or the component corresponding to the $G_2$ factor is as in Example~\ref{ex: hb}. 
Thus, even if it is not true in general that all Lie subalgebras $\hhh$ of $\Lie(G)$ containing $\Lie(U)$ are $T$-stable, the exceptions are well understood, and this will be essential in the proof of Proposition~\ref{prop: connected smooth casev2}.
}
    
\begin{proposition} 	\label{prop: LieH G2}	
Assume that $p = 3$. Let $\mathfrak{h} \subset \Lie(G)$ be a Lie subalgebra containing $\Lie(U)$. 
Suppose that $G = C \times G_{(1)} \times \cdots \times G_{(r)}$ contains at least one simple factor of type $G_2$. 
Then either $\mathfrak{h}$ is $T$-stable, or there exists an index $i$ such that $G_{(i)} \simeq G_2$ and $\mathfrak{h}  \cap \Lie(G_{(i)})=\mathfrak{h}_{\lambda}$ for some $\lambda \in \K^\times$ (see Example~\ref{ex: hb} for the definition of $\mathfrak{h}_{\lambda}$).
\end{proposition}

Before moving on to the proof, recall that the Cartan matrix of $G_2$ gives
\begin{align}
  \label{CartanG2}
  \alpha_1(V_{\alpha_1}) &= 2=-1, 
  & \alpha_2(V_{\alpha_2}) &= 2=-1, \\
  \alpha_1(V_{\alpha_2}) &= 2=-1, 
  & \alpha_2(V_{\alpha_1}) &= -3=0. \nonumber
\end{align}
Also, for the structure constants of $G_2$, we use the table given by \cite[Table 1]{Kol23}, with $\epsilon_1=\epsilon_2=\epsilon_3=\epsilon_4=1$ (Bourbaki's convention).

\begin{proof} 
If there exists an index $i$ such that $G_{(i)} \simeq G_2$ and $\mathfrak{h} \cap \Lie(G_{(i)}) = \mathfrak{h}_{\lambda}$ for some  $\lambda \in \K^\times$, then it is clear that the Lie subalgebra $\hhh \subset \Lie(G)$ is not $T$-stable.  
Suppose, on the contrary, that this condition does not hold. We shall then prove that $\hhh$ is indeed $T$-stable.

\smallskip

We fix an index $i$. For simplicity of notation, we denote the vectors of the Chevalley basis of $G_{(i)}$ simply by $X_\gamma$ and $V_\alpha$, instead of $X_\gamma^i$ and $V_\alpha^i$. Consider a vector $X \in \hhh$ of the form
\[
X \;=\; \sum_{j \neq i} X_j \;+\; \sum_{\gamma \in \Phi^+_i} a_\gamma X_{-\gamma} \;+\; W,
\]
where $X_j \in \Lie(U^-_j)$ for all $j \neq i$, and $W \in \Lie(T)$.  
If $G_{(i)}$ is not of type $G_2$, then, as in the proof of Proposition~\ref{prop: hhh containing LieU}, we deduce that 
\[
\forall \gamma \in \Phi^+_i,\  a_\gamma \neq 0 \;\;\Longrightarrow\;\; X_{-\gamma} \in \hhh.
\]
Assume therefore that $G_{(i)} \simeq G_2$.  
In this case, we shall show that either (once again), for every $\gamma \in \Phi^+_i$,\ $a_\gamma \neq 0$ implies $X_{-\gamma} \in \hhh$, or else
\[
\mathfrak{h} \cap \Lie(G_{(i)}) = \mathfrak{h}_{\lambda}
\quad \text{for some } \lambda \in \K^\times,
\]
which in turn yields the desired result.
Recall that if $G_{(i)}$ is of type $G_2$, then
\[
\Phi^+_i = \{ \alpha_1,\ \alpha_2,\ \alpha_1+\alpha_2,\ 2\alpha_1+\alpha_2,\ 3\alpha_1+\alpha_2,\ 3\alpha_1+2\alpha_2 \}.
\]
Hence we may write $X$ in the form
\[
X \;=\; \sum_{j \neq i} X_j \;+\; aX_{-\alpha_1} \;+\; bX_{-\alpha_2} \;+\; cX_{-\alpha_1-\alpha_2} \;+\; dX_{-2\alpha_1-\alpha_2} \;+\; eX_{-3\alpha_1-\alpha_2} \;+\; fX_{-3\alpha_1-2\alpha_2} \;+\; W,
\]
where at least one coefficient $a,\dots,f \in \K$ is nonzero.  
We now distinguish cases according to
\[
l := \max \{\, l(\gamma) \;\mid\; \gamma \in \Phi^+_i \ \text{with } a_\gamma \neq 0 \,\} \;\in \{1,2,3,4,5\},
\]
where $l(\gamma)$ denotes the length of the root $\gamma$.  
Note that the case $l=5$ will be considered before the case $l=4$, since only in the latter does the Lie subalgebra $\hhh_\lambda$ appear.

\smallskip

\textbf{$\bullet$ Case $l=1$.}  
Consider
\[
X \;=\; \sum_{j \neq i} X_j \;+\; aX_{-\alpha_1} \;+\; bX_{-\alpha_2} \;+\; W,
\quad (a,b) \neq (0,0).
\]
Assume first that $a \neq 0$. Then
\[
[X, X_{\alpha_1}] \;=\; aV_{\alpha_1} \;+\; \alpha_1(W) X_{\alpha_1} \;\in\; \hhh.
\]
Using again the hypothesis that $\Lie(U) \subset \hhh$, we deduce that $V_{\alpha_1} \in \hhh$, and hence also
\[
[X, V_{\alpha_1}] \;=\; -aX_{-\alpha_1} \;\in\; \hhh.
\]
If moreover $b=0$, the claim follows. Otherwise, if $b \neq 0$, we compute
\[
[X - aX_{-\alpha_1}, X_{\alpha_2}] \;=\; bV_{\alpha_2} \;+\; \alpha_2(W) X_{\alpha_2} \;\in\; \hhh,
\]
which implies $V_{\alpha_2} \in \hhh$, and finally
\[
[X - aX_{-\alpha_1}, V_{\alpha_2}] \;=\; -bX_{-\alpha_2} \;\in\; \hhh.
\]
Thus, in either case, $X_{-\alpha_1}, X_{-\alpha_2} \in \hhh$.  
If instead $a=0$ and $b \neq 0$, then
\[
[X, X_{\alpha_2}] \;=\; bV_{\alpha_2} \;+\; \alpha_2(W) X_{\alpha_2} \;\in\; \hhh,
\]
and we proceed analogously to conclude that $X_{-\alpha_2} \in \hhh$.

\smallskip

\textbf{$\bullet$ Case $l=2$.}  
Up to rescaling, we may assume that $c=1$, so that 
\[
X \;=\; \sum_{j \neq i} X_j \;+\; aX_{-\alpha_1} \;+\; bX_{-\alpha_2} \;+\; X_{-\alpha_1-\alpha_2} \;+\; W.
\]
By the result of Case~$l=1$, it suffices to show that $X_{-\alpha_1-\alpha_2} \in \hhh$.  
We first take brackets with $X_{\alpha_2}$, and then with $X_{\alpha_1}$:
\[
\hhh \ni [[X, X_{\alpha_2}], X_{\alpha_1}]
\;=\; [\, bV_{\alpha_2} - X_{-\alpha_1} \;+\; \alpha_2(W) X_{\alpha_2}, \; X_{\alpha_1}]
\;=\; -bX_{\alpha_1} - V_{\alpha_1} - \alpha_2(W)X_{\alpha_1+\alpha_2}.
\]
Since $\Lie(U) \subset \hhh$ by assumption, it follows that $V_{\alpha_1} \in \hhh$.  
Now consider
\[
\hhh \ni Y \;:=\; [X, V_{\alpha_1}]
\;=\; -aX_{-\alpha_1} \;+\; (\alpha_1+\alpha_2)(V_{\alpha_1})\, X_{-\alpha_1-\alpha_2}
\;=\; -aX_{-\alpha_1} \;-\; X_{-\alpha_1-\alpha_2}.
\]
If $a=0$, then $Y = -X_{-\alpha_1-\alpha_2} \in \hhh$, and we are done.  
If instead $a \neq 0$, then
\[
[Y, X_{\alpha_2}] \;=\; X_{-\alpha_1} \;\in\; \hhh,
\]
which implies
\[
X_{-\alpha_1-\alpha_2} \;=\; -Y - aX_{-\alpha_1} \;\in\; \hhh.
\]
Thus in either case, $X_{-\alpha_1-\alpha_2} \in \hhh$, as required.

\smallskip

\textbf{$\bullet$ Case $l=3$.} Up to rescaling, let us assume $d=1$ and consider
\[
X = \sum_{j \neq i} X_j + aX_{-\alpha_1} + bX_{-\alpha_2} + cX_{-\alpha_1-\alpha_2} + X_{-2\alpha_1-\alpha_2} + W.
\]
Thanks to the previous cases, it suffices to show that $X_{-2\alpha_1-\alpha_2} \in \hhh$.  
From now on, set 
\[
Z := [X_{-\alpha_1-\alpha_2},\,X_{\alpha_1+\alpha_2}] \in \Lie(T).
\]  
Taking successive brackets of $X$ first with $X_{\alpha_1+\alpha_2}$ and then with $X_{\alpha_1}$, we obtain
\begin{align*}
\hhh &\ni [X,X_{\alpha_1+\alpha_2}]= -bX_{\alpha_1} + cZ + X_{-\alpha_1} + (\alpha_1+\alpha_2)(W)X_{\alpha_1+\alpha_2},\ \text{and so} \\ 
\hhh &\ni  [cZ + X_{-\alpha_1}, X_{\alpha_1}]=c[Z,X_{\alpha_1}]+V_{\alpha_1}=c\alpha_1(Z)X_{\alpha_1}+V_{\alpha_1}.
\end{align*}
We deduce that $V_{\alpha_1} \in \hhh$. 
Consequently, 
\begin{align*}
\hhh \ni Y' &:= X + [V_{\alpha_1},X] \\
&= X +aX_{-\alpha_1} +c X_{-\alpha_1-\alpha_2} -X_{-2\alpha_1-\alpha_2} \\
&= \sum_{j \neq i} X_j - aX_{-\alpha_1} + bX_{-\alpha_2} -cX_{-\alpha_1-\alpha_2} + W.
\end{align*} 
Now, applying to the vector $Y'$ one of the previous steps (where the length is at most two), we conclude that 
\[
aX_{-\alpha_1},\quad bX_{-\alpha_2},\quad cX_{-\alpha_1-\alpha_2} \;\in\; \hhh.
\]
Finally,  
\[
 X_{-2\alpha_1-\alpha_2}=[\,X - aX_{-\alpha_1} - bX_{-\alpha_2} - cX_{-\alpha_1-\alpha_2},\, V_{\alpha_1}]  \in \hhh,
\]
which completes the proof.

	\smallskip

\textbf{$\bullet$ Case $l=5$.} Up to rescaling, let us assume $f=1$. In this case, we claim that 
\[
\mathfrak h \cap \Lie(G_{(i)}) = \Lie(G_{(i)}).
\]
First, we have
\[
\hhh \ni  [X, X_{\alpha_1+\alpha_2}] 
= -bX_{\alpha_1} + cZ + dX_{-\alpha_1} - X_{-2\alpha_1-\alpha_2} 
   + (\alpha_1+\alpha_2)(W)\, X_{\alpha_1+\alpha_2},
\]
which, by the argument in the case $l=3$, implies that $X_{-2\alpha_1-\alpha_2} \in \mathfrak h$.  
Next,
\[
\hhh \ni X' := [cZ+dX_{-\alpha_1}-X_{-2\alpha_1-\alpha_2},\, X_{\alpha_1}] 
= c\,\alpha_1(Z)\, X_{\alpha_1} + d V_{\alpha_1} + X_{-\alpha_1-\alpha_2},
\]
which, using the proof for $l=2$, implies that $X_{-\alpha_1-\alpha_2} \in \mathfrak h$.  
Furthermore,
\[
\hhh \ni [X', X_{\alpha_2}] = c\,\alpha_1(Z)\, X_{\alpha_1+\alpha_2} - X_{-\alpha_1},
\]
so $X_{-\alpha_1} \in \mathfrak h$.  
We are therefore left with
\[
\hhh \ni X'' := \sum_{j \neq i} X_j + bX_{-\alpha_2}+ e X_{-3\alpha_1-\alpha_2}+X_{-3\alpha_1-2\alpha_2} + W.
\]
Then,
\begin{align*}
\hhh &\ni [X'',X_{3\alpha_1+2 \alpha_2}] 
      = b X_{3\alpha_1+ \alpha_2} -e X_{\alpha_2}+[X_{-3\alpha_1-2 \alpha_2},X_{3\alpha_1+2 \alpha_2}]-\alpha_2(W)X_{3\alpha_1+2\alpha_2},\ \text{and so} \\
\hhh &\ni Z':=[X_{-3\alpha_1-2 \alpha_2},X_{3\alpha_1+2 \alpha_2}]
      = -V_{\alpha_1}+V_{\alpha_2},
\end{align*}
which yields $V_{\alpha_2} \in \hhh$. Moreover,
\[
\hhh \ni [X'',X_{\alpha_2}] = bV_{\alpha_2} + X_{-3\alpha_1-\alpha_2} +\alpha_2(W) X_{\alpha_2},
\]
which implies that $X_{-3\alpha_1-\alpha_2} \in \hhh$.  
Next,
\[
\hhh \ni [X''-eX_{-3\alpha_1-\alpha_2}, X_{3\alpha_1+\alpha_2}] 
= - X_{-\alpha_2} + \alpha_2(W)\,X_{3\alpha_1+\alpha_2},
\]
so $X_{-\alpha_2} \in \hhh$. Finally,
\[
\hhh \ni [X_{-3\alpha_1-\alpha_2},X_{-\alpha_2}] =  X_{-3\alpha_1-2\alpha_2},
\]
and thus we conclude that $\hhh = \Lie(G_{(i)})$, as claimed.

\smallskip

	\textbf{$\bullet$ Case $l=4$.} Up to rescaling, let us assume that $e = 1$ (and $f=0$ since $l=4$) and consider
\[
X = \sum_{j \neq i} X_j + a X_{-\alpha_1} + b X_{-\alpha_2} + c X_{-\alpha_1-\alpha_2} + d X_{-2\alpha_1-\alpha_2} + X_{-3\alpha_1-\alpha_2}+W.
\]
We aim to show that either $\hhh$ is $T$-stable, or there exists $\lambda \in \K^\times$ such that $\hhh \cap \Lie(G_{(i)}) = \hhh_\lambda$.

Assume first that $d \neq 0$. Then
\begin{align*}
\hhh &\ni [X, X_{\alpha_1+\alpha_2}]=- b X_{\alpha_1} + c Z + d X_{-\alpha_1} + (\alpha_1+\alpha_2)(W) X_{\alpha_1+\alpha_2},\ \text{and so} \\ 
\hhh &\ni  [c Z + d X_{-\alpha_1} , X_{\alpha_1}]=c\alpha_1(Z)X_{\alpha_1}+d V_{\alpha_1},
\end{align*}
which implies $V_{\alpha_1} \in \hhh$ (since $d \neq 0$). Consequently, 
\[
\hhh \ni [V_{\alpha_1}, X] = a X_{-\alpha_1} + c X_{-\alpha_1-\alpha_2} - d X_{-2\alpha_1-\alpha_2}. 
\] 
By the argument in the case $l=3$, we can conclude that $a X_{-\alpha_1}, c X_{-\alpha_1-\alpha_2}, X_{-2\alpha_1-\alpha_2} \in \hhh$. Moreover, 
\begin{align*}
&\hhh \ni X_{-\alpha_1}=[X_{-2\alpha_1-\alpha_2},X_{\alpha_1+\alpha_2}], \text{ and}\\
&\hhh \ni X_{-\alpha_1-\alpha_2}=[X_{\alpha_1},X_{-2\alpha_1-\alpha_2}].
\end{align*}
We are therefore left with
\[
\hhh \ni Y' := \sum_{j \neq i} X_j + bX_{-\alpha_2}+  X_{-3\alpha_1-\alpha_2} + W.
\]
Then,
\begin{align*}
&\hhh \ni Z'':=[Y',X_{3\alpha_1+ \alpha_2}]=[X_{-3\alpha_1- \alpha_2},X_{3\alpha_1+ \alpha_2}]+\alpha_2(W)X_{3 \alpha_1+\alpha_2},\ \text{and so}\\
&\hhh \ni  [X_{-3\alpha_1- \alpha_2},X_{3\alpha_1+ \alpha_2}]=-V_{\alpha_1}-V_{\alpha_2},
\end{align*}
which yields $V_{\alpha_2} \in \hhh$. In particular, $\hhh \cap \Lie(G_{(i)})$ contains $\Lie(T_{(i)})$. 
It follows that
\[
\hhh \ni [V_{\alpha_2},Y']=bX_{-\alpha_2}+  X_{-3\alpha_1-\alpha_2}.
\]
If $b=0$, then we are done. 
Otherwise, setting $\lambda = b \in \K^\times$, we obtain that 
\[X_\lambda:=\lambda X_{-\alpha_2} + X_{-3\alpha_1-\alpha_2} \in \hhh.\] 
(Notice that $X_\lambda$ is exactly the vector considered in \Cref{ex: hb}.)
We deduce that $\hhh \cap \Lie(G_{(i)})$ contains $\hhh_\lambda$.
If the inclusion is strict, then either $\hhh$ contains a length-five vector, which implies (by the case $l=5$) that $\hhh \cap \Lie(G_{(i)})= \Lie(G_{(i)})$, or
\[
\K X_{-\alpha_2} \oplus \K X_{-3\alpha_1-\alpha_2} \subset \hhh.
\]
In this last case, taking the bracket of $X_{-\alpha_2}$ and $X_{-3\alpha_1-\alpha_2}$ again yields $\hhh \cap \Lie(G_{(i)})= \Lie(G_{(i)})$.

Now, returning to the beginning, assume instead that $d=0$.  
Taking the bracket of $X$ with $X_{2\alpha_1+\alpha_2}$ again shows that $X_{-\alpha_1}$, 
and hence $V_{\alpha_1}$, lie in $\hhh$.  
Proceeding as above, we deduce that $\hhh$ also contains $X_{-2\alpha_1-\alpha_2}$ and $X_{-\alpha_1-\alpha_2}$, 
and the remainder of the proof is analogous to the argument given above (when $Y'$ appears).  
This completes the proof of the proposition.
\end{proof}

\section{Strongly horospherical subgroups and classification for \texorpdfstring{$p \geq 3$}{p>2}} \label{sec: strongly horo subgroups}
Let $G$ be a connected reductive group. 
In this section, we establish some basic properties of strongly horospherical subgroups of $G$, and then provide a complete classification in the case $p \geq 3$ (Corollary~\ref{cor: classification strongly horo}).

\subsection{Structure of the quotient group}
Let $H$ be a strongly horospherical subgroup of $G$.  
Over a field of characteristic $0$, the quotient group $N_G(H)/H$ is always a torus. We will see that this is no longer necessarily the case over a field of characteristic $2$.

\begin{lemma}
\label{lem heightr}
    Let $H$ be a strongly horospherical subgroup of $G$ with normalizer $P$. Let $\gamma$ be a positive root of $G$. {Let $R$ be a $\mathbb{K}$-algebra, $r \in \N \cup \{\infty\}$ and consider some $x \in \boldsymbol{\alpha}_{p^r}(R)$ such that that $u_{-\gamma} (x) \in P(R)$.
    Then we have that $u_{-\gamma}(x^2) \in H(R)$.}
    \end{lemma}
    
\begin{proof}
    Since the derived subgroup of $G$ is simply-connected, it is enough to work inside of the copy of $\SL_2$ generated by $U_\gamma$ and $U_{-\gamma}$, so we can assume that $G=\SL_2$ and that $\gamma$ is the only positive root. {Let $t \in R^\times=\Gm(R)$, then by the assumption that $u_{-\gamma}(x) \in P(R)$, we have that
    \[
     H(R) \ni \begin{bmatrix}
        t & 0 \\
        x & t^{-1}
    \end{bmatrix} 
    \begin{bmatrix}
        1 & 1 \\
        0 & 1 
    \end{bmatrix}
    \begin{bmatrix}
        t^{-1} & 0\\
        -x & t
    \end{bmatrix} = \begin{bmatrix} 1-tx & t^2\\ -x^2 & 1+tx\end{bmatrix}.
    \]
    Since this is true for arbitrary $t$ (as $H$ is assumed to be \emph{strongly} horospherical), we can take the limit for $t \to 0$. (More specifically, the homomorphism $\Gm \to H$ sending $t$ to the above matrix can be extended to a scheme morphism $\mathbf{A}^1 \to H$, the value at $0$ is what we call its limit: for more details, see \cite[Section 13.b]{Milne}.) This yields that $u_{-\gamma} (-x^2) \in H(R)$ and concludes the proof.}
\end{proof}

\begin{lemma}
\label{lem heights}
    Let $H$ be a strongly horospherical subgroup of $G$ with normalizer $P$, and let $\gamma$ be a positive root of $G$. Then
\[
\left\{
    \begin{array}{ll}
       \height(H \cap U_{-\gamma}) \leq \height(P \cap U_{-\gamma}) \leq \height(H \cap U_{-\gamma})+1 & \text{ if } p =2; \\
       H \cap U_{-\gamma} = P \cap U_{-\gamma} & \text{ if } p \geq 3.
    \end{array}
\right.    
\]    
\end{lemma}
 
\begin{proof}
    The inequality $\height(H \cap U_{-\gamma}) \leq \height(P \cap U_{-\gamma})$ is clear. {  
To prove the reverse inequality, let $r \in \N \cup \{\infty\}$ denote the height of $P \cap U_{-\gamma}$.
Set $R=\mathbb{K}[y]/(y^{p^r})$, and let
$x \in \boldsymbol{\alpha}_{p^r}(R)\simeq (P \cap U_{-\gamma})(R)$.
Since $u_{-\gamma}(x)\in (P \cap U_{-\gamma})(R)$, \Cref{lem heightr} implies that
$u_{-\gamma}(x^2)\in (H \cap U_{-\gamma})(R)$.
Suppose that the height of $H \cap U_{-\gamma}$ is strictly smaller than $r$. Then
\[
(x^2)^{p^{r-1}}=x^{2p^r}=0 \in \boldsymbol{\alpha}_{p^r}(R).
\]
In particular, $y^{2p^{r-1}}=0$ in $R$.
If $p\ge 3$, this is impossible, since the nilpotence index of $y$ in $R$ is $p^r$.
Therefore $H \cap U_{-\gamma}$ also has height $r$.
If $p=2$, the above argument only yields $\height(P \cap U_{-\gamma})
\le \height(H \cap U_{-\gamma})+1$.    }
\end{proof}

\begin{remark}\label{rk: P=HT for strongly horo}
Assume that $p \geq 3$.  
Let $H$ be a strongly horospherical subgroup of $G$.  
By Lemma~\ref{lem heights}, we have $P = N_G(H) = HT$.  
In particular, if $H$ is smooth, then $P$ is smooth as well.  
\end{remark}

\begin{proposition}\label{prop P/H is a torus}
Let $H$ be a strongly horospherical subgroup of $G$ with normalizer $P$.
The following hold:
\begin{enumerate}
\item\label{item 1 prop torus} If $p \geq 3$, then $\T:=P/H$ is a torus.
\item\label{item 2 prop torus} If $p=2$, then $P/H \simeq \T \ltimes V$, where $\T:=(P/H)_{\red}$ is a torus and $V$ is an infinitesimal  unipotent group of height at most $1$.
\end{enumerate}
\end{proposition}
 
\begin{proof}
\ref{item 1 prop torus}: Assume first that $p \geq 3$. Then $P/H$ is a torus since $H \cap U=P \cap U=U$ and $H \cap U_{-}=P \cap U_{-}$ by Lemma~\ref{lem heights}.\\
\ref{item 2 prop torus}: Assume now that $p=2$. Then, again by Lemma~\ref{lem heights}, $\T:=(P/H)_{\red}$ is a torus and there is a split exact sequence
\[
1 \to V \to P/H \to \T \to 1,
\]
where $V$ is an infinitesimal  unipotent group of height at most $1$.
\end{proof}

\begin{remark}
Let $H$ be a strongly horospherical subgroup with normalizer $P$.
If $P$ is smooth, then $P/H$ is smooth. In particular, $P/H$ is a torus. 
\end{remark}

\begin{example}\label{ex: not always a torus when p=2}
Assume $p=2$. Let $G=\SL_2$ and \[
H \defeq \left\{
    \begin{bmatrix}
        a & b \\ 0 & d
    \end{bmatrix} \ \middle|\   a^2=1
    \right\} \subset \SL_2.
 \ \text{(This is the subgroup $\mathcal{A}(2,\infty)$ in \cite[Lemma 3.1]{Knop}.)}\]
Then $H$ is a strongly horospherical subgroup of $G$, $P=N_G(H)=F^{-1}(B)$, and $\height(P \cap U_{-\gamma}) = \height(H \cap U_{-\gamma})+1=2$. 
In particular, $P/H \simeq \Gm \ltimes \alpha_2$ is not a torus.
\end{example}

\subsection{Classification for \texorpdfstring{$p \geq 3$}{p>2}} \label{sec: classification strongly horo}
By assuming that the characteristic of the base field $\K$ is not equal to $2$, we get a classification result for strongly horospherical subgroups of $G$ that is analogous to the well-known one in characteristic $0$: all such subgroups are obtained by taking kernels of characters on a parabolic subgroup of $G$.

\begin{proposition}\label{prop characterization strongly horo car p>2}
Assume that $p \geq 3$. 
\begin{enumerate}
\item \label{prop characterization strongly horo car p>2 item i} Let $H$ be a strongly horospherical subgroup of $G$ with normalizer $P$. 
Then $H=\bigcap_{\chi \in M} \ker(\chi_i)$, where $M$ is a sublattice of the character lattice $X^\ast(P)$.
\item \label{prop characterization strongly horo car p>2 item ii} Conversely, given a parabolic subgroup scheme $P$ of $G$ and a sublattice $M$ of the character lattice $X^\ast(P)$, the scheme-theoretic intersection $H:=\bigcap_{\chi \in M} \ker(\chi_i)$ is a strongly horospherical subgroup of $G$ such that $N_G(H)=P$.
\end{enumerate}
\end{proposition}

\begin{proof}
\ref{prop characterization strongly horo car p>2 item i}: {
By Proposition~\ref{prop P/H is a torus}, the quotient $\T=P/H$ is a torus. Fix an isomorphism of algebraic groups
$\T \simeq \Gm^r$, with $r \geq 0$, and let $\chi_i \colon P \to \Gm$
denote the element of the character lattice $X^\ast(P)$ obtained by composing the quotient morphism
$P \to \T \simeq \Gm^r$ with the $i$-th projection $\Gm^r \to \Gm$. Let $M$ be the rank-$r$ sublattice of $X^\ast(P)$ generated by the $\chi_i$. Then $H=\bigcap_{\chi \in M} \ker(\chi)$.}\\
\ref{prop characterization strongly horo car p>2 item ii}: Let \( M \) be a sublattice of the character lattice \( X^\ast(P) \), and define \( H := \bigcap_{\chi \in M} \ker(\chi) \). For every \( \chi \in X^\ast(P) \), we have \( U \subset \ker(\chi) \), which implies that \( H \) is horospherical. Furthermore, it is clear that \( P \subset N_G(H) \). Hence, \( Q := N_G(H) \) is a parabolic subgroup of \( G \).
Now, for every positive root \( \gamma \) of \( G \), consider \( U_{-\gamma} \). We have 
\[
U_{-\gamma} \cap Q = U_{-\gamma} \cap H \subset U_{-\gamma} \cap P,
\]
where the first equality follows from Lemma~\ref{lem heights}. Therefore, \( Q \subset P \). It follows that \( N_G(H) = P \).
\end{proof}
 
\begin{remark} \label{rk: classification in the case p=2}
Assume that the base field has characteristic $p=2$.  
Let $H$ be a strongly horospherical subgroup with normalizer $P$.  
Then it may happen that $HT \subsetneq P$, yet the quotient $HT/H$ is still a torus:
\[
HT/H \simeq \mathbb{T} := (P/H)_{\mathrm{red}}.
\]
Thus $H$ is always obtained as the intersection of the kernels of characters of the parabolic subgroup $HT$, 
whereas its normalizer $P=N_G(H)$ can be larger, the excess being measured precisely by the unipotent infinitesimal group $V$ defined in Proposition~\ref{prop P/H is a torus}\eqref{item 2 prop torus}.
\end{remark}

Let $\mathcal{P}$ denote the set of functions $\Phi^+ \to \mathbb{N} \cup \{\infty\}$ associated with parabolic subgroups of $G$ (see Section~\ref{sec: associated function}). By Theorem~\ref{thm:numericalfunction}, there is a one-to-one correspondence between the set of conjugacy classes of parabolic subgroups of $G$ and the set $\mathcal{P}$. This parametrization can be naturally extended to strongly horospherical subgroups of $G$.

\begin{corollary}\label{cor: classification strongly horo}
Assume that $p \geq 3$. 
There is a one-to-one correspondence between the set of conjugacy classes of strongly horospherical subgroups of $G$ and the set of pairs $(\varphi, M)$, where $\varphi \in \mathcal{P}$ and $M$ is a sublattice of $X^\ast(P_\varphi)$, with $P_\varphi$ the standard parabolic subgroup of $G$ associated with the function $\varphi$.
\end{corollary}

\begin{proof}
	Given a standard parabolic subgroup $P_\varphi$ of $G$ and a sublattice $M \subset X^\ast(P_\varphi)$, we define a strongly horospherical subgroup $H \subset G$ as the scheme-theoretic intersection
\[
H := \bigcap_{\chi \in M} \ker(\chi).
\]
We then map the pair $(\varphi,M)$ to the conjugacy class of $H$.

Conversely, given a conjugacy class of strongly horospherical subgroups of $G$, we fix the representative $H$ containing the unipotent radical $U$ of $B$. We define the standard parabolic subgroup $P := N_G(H)$. Then there exists a unique element $\varphi \in \mathcal{P}$ such that $P = P_\varphi$. We define $M$ as the sublattice of $X^\ast(P)$ consisting of all characters $\chi \in X^\ast(P)$ that vanish on $H$, i.e., those that factor through the quotient torus $\mathbb{T} = P/H$.

It is then straightforward, using Proposition~\ref{prop characterization strongly horo car p>2}, to verify that these two constructions are well-defined and inverses of each other.
\end{proof}

\begin{remark} Assume that $p \geq 3$. 
The strongly horospherical subgroup $H$, corresponding to a pair $(\varphi,M)$ as above, is smooth if and only if $P_\varphi$ is smooth and the quotient lattice $X^\ast(T)/M$ has no $p$-torsion.
\end{remark}

\begin{example}
It is possible for a strongly horospherical subgroup $H \subset G$ to be non-smooth while $P$ is smooth. For instance, consider the case where $G = \SL_2$, $q=p^t \geq 3$, and define 
\[
H :=\left\{
    \begin{bmatrix}
        a & b \\ 
        0 & d
    \end{bmatrix}
    \in \SL_2 \ \middle|\  a^q = 1
\right\}.
\]
(This subgroup corresponds to \(\mathcal{A}(q,\infty)\) in \cite[Lemma 3.1]{Knop}.) Then, $H$ is a strongly horospherical non-smooth subgroup of $G$ while $P = N_G(H) = B$ is smooth. In this case, the quotient lattice is $X^\ast(T)/M \simeq \Z/q\Z$.
\end{example}

\section{Horospherical subgroups are strongly horospherical when \texorpdfstring{$p \geq 3$}{p>2}}\label{sec: main section}
Let $G$ be a reductive group, which we may assume, without loss of generality, to be of the form $G = C \times G'$, where $C$ is a torus and $G' = \prod_{j \in J} G_{(j)}$ is a simply connected semisimple group.
This section is dedicated to proving Theorem~\ref{th: general case}, i.e.~to proving that horospherical subgroups of $G$ are strongly horospherical when $p \geq 3$.

\smallskip

\subsection{Subgroups of height at most one} \label{sec: Subgroups of height at most one}
Let us recall a fundamental structure theorem describing finite subgroup schemes of height at most one. 
It will be useful in the proofs of Proposition~\ref{prop: connected smooth casev2} and Theorem~\ref{th: general case}.

\begin{theorem}[{\emph{cf.~\cite[II, \S~7, n°4]{DG}}}]
\label{th: DG height one}
Let $G$ be a linear algebraic group over $\K$. Then:
\begin{enumerate}
  \item \label{th: DG height one item i} The map $J \mapsto \Lie(J)$ induces a bijection between the set of subgroup schemes of ${}_1G$ and the set of restricted Lie subalgebras of $\Lie(G)$. Under this correspondence, normal subgroup schemes correspond to  restricted Lie ideals.
  
  \item \label{th: DG height one item ii} Let $J$ and $J'$ be subgroup schemes {of ${}_1G$.} Then $J \subset J'$ if and only if $\Lie(J) \subset \Lie(J')$.
  
  \item If $f_1$ and $f_2$ are homomorphisms with source ${}_1G$, then $f_1 = f_2$ if and only if $Df_{1,e} = Df_{2,e}$.
  
  \item If $J$ is a subgroup scheme {of ${}_1G$}, then {there is an equality of restricted Lie algebras}
  \[
  \Lie(N_G(J)) = N_{\Lie(G)}(\Lie(J)),
  \]
  where $N_G(J)$ denotes the normalizer of $J$ in $G$, and $N_{\Lie(G)}(\Lie(J))$ the normalizer of the Lie subalgebra $\Lie(J)$ in $\Lie(G)$.
  
  \item Let $V$ be a finite-dimensional representation of $G$. Then a vector subspace of $V$ is ${}_1G$-stable if and only if it is $\Lie(G)$-invariant.
\end{enumerate}
\end{theorem}

\subsection{Warm-up: the \texorpdfstring{$\SL_2$}{SL(2)} case} \label{sec: SL2 case}
We now consider the simplest case: all subgroup of $\SL_2$ have been classified by Knop in \cite{Knop}. 
Out of this classification, we easily deduce the classification of the horospherical subgroups of $\SL_2$. 

\begin{theorem}[{see \cite[Theorem 5.1 and Lemma 3.1]{Knop}}] 	\label{thm: Knop horo subgroups SL2}
	Assume that $p \geq 2$.
Let $H$ be a horospherical subgroup of $G = \SL_2$. Then $H$ is conjugate to the pullback by an iterated Frobenius of one of the following:
\begin{itemize}
    \item $G$, or $B$, or $\boldsymbol{\mu}_n \ltimes U$ for some $n \geq 1$, in which case the normalizer of the group contains $B$.
    \item If $p = 2$, the subgroup $\mathcal{B}(\infty)$ (see Example~\ref{Binfinity}), in which case the group coincides with its normalizer.
\end{itemize}
\end{theorem}

The next observation will be useful in the proof of Theorem~\ref{th: general case}.

\begin{corollary} \label{cor: Knop horo subgroups SL2}
Assume that $p \geq 3$. Then every horospherical subgroup of $\SL_2$ is strongly horospherical.
\end{corollary}

\subsection{The smooth case}
In this section, we prove Theorem~\ref{th: general case}, first under the additional assumption that $H$ is smooth and connected (Proposition~\ref{prop: connected smooth casev2}~\ref{prop: connected smooth casev2 item ii}), and then under the weaker assumption that $H$ is merely smooth (corollary~\ref{cor: smooth case}).

\begin{proposition}\label{prop: connected smooth casev2}
\label{prop:LieH-T-stable}
Assume that $p \geq 3$. Let $H$ be a horospherical subgroup of $G$. Then:
\begin{enumerate}
    \item\label{prop: connected smooth casev2 item i} The Lie algebra $\Lie(H)$ is $T$-stable (for the restricted adjoint action $T \curvearrowright \Lie(G)$).
    \item\label{prop: connected smooth casev2 item ii} If $H$ is furthermore smooth and connected, then $H$ is strongly horospherical.
\end{enumerate}
\end{proposition}

\begin{proof}
If $H$ is a connected smooth horospherical subgroup, then it is strongly horospherical if and only if $\Lie(H)$ is $T$-stable. 
Hence it suffices to prove \eqref{prop: connected smooth casev2 item i}.

If $p \geq 5$, or if $G$ does not contain any simple factor of type $G_2$, then the result follows from Proposition~\ref{prop: hhh containing LieU}. Thus, we may assume that $p = 3$ and that $G$ contains at least one simple factor of type $G_2$; fix such a factor and denote it by $G_{(i)}$.

By Proposition~\ref{prop: LieH G2}	, it is enough to show that $\Lie(H) \cap \Lie(U^-_i)$ cannot be of the form $\hhh_\lambda \cap \Lie (U^-_i)$ for any $\lambda \in \K^\times$ (see Example~\ref{ex: hb} for the definition of $\hhh_\lambda$). Assume, for contradiction, that this is the case, i.e.~that $\Lie(H) \cap \Lie (U^-_i)= \hhh_\lambda \cap \Lie (U^-_i)$ for some $\lambda \in \K^\times$.

The kernel $K_{G_{(i)}}$ of the very special isogeny of $G_{(i)}$ is a subgroup of height 1 whose Lie algebra is contained  in $\hhh_\lambda \subset \Lie(H)$.
Hence, by Theorem~\ref{th: DG height one}~\eqref{th: DG height one item ii}, the kernel $K_{G_{(i)}}$ is contained in $H$.
Now consider the quotient group:
\[
\overline{H} := H / K_{G_{(i)}} \subset C \times \left( \prod_{j \neq i} G_{(j)} \right) \times \overline{G_{(i)}},
\]
where $\overline{G_{(i)}}$ is the image of $G_{(i)}$ under the very special isogeny and is again a simple group of type $G_2$.
By construction, $\overline{H}$ is a horospherical subgroup. Then by Proposition~\ref{prop: LieH G2} applied to $\Lie(\overline{H})$, we again have two possibilities:
\begin{itemize}
    \item either $\Lie(\overline{H}) \cap \Lie(\overline{U}^-_i)$ is $\overline{T}$-stable;
    \item or there exists $\mu \in \K^\times$ such that $\Lie(\overline{H}) \cap \Lie(\overline{U}^-_i) = \overline{\hhh}_\mu \cap \Lie(\overline{U}^-_i)$.
\end{itemize}

In the first case, since $\Lie(\overline{H}) \cap \Lie(\overline{U}^-_i)$ is $\overline{T}$-stable, pulling back by the very special isogeny gives that $\Lie(H) \cap \Lie(U^-_i)$ is also $T$-stable, contradicting the assumption that $\Lie(H) \cap \Lie(U^-_i) = \hhh_\lambda \cap \Lie(U^-_i)$.

In the second case, we observe that the kernel $K_{\overline{G_{(i)}}}$ of the very special isogeny of $\overline{G_{(i)}}$ is then contained in $\overline{H}$:
\begin{equation}
\label{K contained}
K_{\overline{G_{(i)}}} \subset \overline{H} = H / K_{G_{(i)}}.
\end{equation}
Recall that the composition of the two successive very special isogenies equals the Frobenius morphism (see \cite[Definition 7.1.3]{CGP15}):
\[
\pi_{\overline{G}_{(i)}} \circ \pi_{G_{(i)}} = F_{G_{(i)}}.
\]
Thus, the inclusion \eqref{K contained} implies that the full Frobenius kernel ${}_1 G_{(i)}$ is contained in $H$. 
In particular,
\[
\Lie(H) \cap \Lie(U^-_i) = \Lie(U^-_i),
\]
which contradicts the assumption that this intersection equals $\hhh_\lambda \cap \Lie(U^-_i)$. 
Hence, in all cases, $\Lie(H)$ is $T$-stable, which proves \eqref{prop: connected smooth casev2 item ii}, and therefore also \eqref{prop: connected smooth casev2 item i}.
\end{proof}

\begin{corollary} \label{cor: smooth case}
Assume that $p \geq 3$. Let $H$ be a smooth horospherical subgroup of $G$. Then $H$ is strongly horospherical.
\end{corollary}

\begin{proof}
Since $H$ is horospherical, $H^\circ$ is also horospherical, and thus strongly horospherical by Proposition~\ref{prop: connected smooth casev2}~\eqref{prop: connected smooth casev2 item ii}. Hence, by Proposition~\ref{prop P/H is a torus}\,\eqref{item 1 prop torus}, the quotient group $N_G(H^\circ)/H^\circ$ is a torus. 
In particular, the quotient variety
\[
N_G(H^\circ)/H = \left(N_G(H^\circ)/H^\circ\right) \big/ \left(H/H^\circ\right)
\]
is again a torus. That is, $N_G(H^\circ)$ normalizes $H$, which implies that $N_G(H) = N_G(H^\circ)$ is a parabolic subgroup of $G$, and thus $H$ is strongly horospherical.
\end{proof}

\subsection{The general case}
We are now ready to prove the main result of this article:

\begin{theorem} \label{th: general case}
Assume that $p \geq 3$. Let $H$ be a  horospherical subgroup of $G$. Then $H$ is strongly horospherical.
\end{theorem}

\begin{proof}
To simplify the argument, we assume that $G$ is simple. The general case is similar, but slightly more technical to write down.
We proceed by induction on the height of $H$. 
If the height of $H$ is equal to $0$, then $H$ is smooth, and we conclude by Corollary~\ref{cor: smooth case}. 

\smallskip

Assume now that the statement holds for any horospherical subgroup of height at most $m \in \mathbb{N}_{\geq 0}$. Let $H$ be a horospherical subgroup of height $m+1$. Then $P := HT$ is a parabolic subgroup of $G$, and we will prove that $P = N_G(H)$. 

By Theorems~\ref{thm classification parabolic subgroups} and~\ref{thm classification parabolics 2}, there exists a subset $I \subset \Delta$ such that
\begin{equation}\label{eq: description of P in the proof of the connected case}
P = \bigcap_{\nu \in \Delta \setminus I} Q^\nu \quad \text{with} \quad Q^\nu = {}_{r_\nu} G P^\nu \quad \text{or} \quad Q^\nu = (F^{r_\nu - 1})^{-1}(K) P^\nu,
\end{equation}
where the second case occurs only if $p = 3$ and $G$ is of type $G_2$. Here, $K$ denotes the kernel of the very special isogeny $\pi_G\colon G \to \overline{G}$ (see Section~\ref{sec: very special isogeny}); \( r_\nu \geq 0 \) in the first case, and \( r_\nu \geq 1 \) in the second. The subgroup $P^\nu$ denotes the maximal reduced parabolic subgroup of $G$ associated with the simple root~$\nu$. Then, as discussed in \Cref{sec character lattice of parab subgroups}, the character group of~$P$ is of the form
\[
X^\ast(P) = \bigoplus_{\nu \in \Delta \setminus I} \mathbf{Z} \, p^{r_\nu} \varpi_\nu.
\]

Note that $P_{\mathrm{red}} = H_{\mathrm{red}}T = N_G(H_{\mathrm{red}})$, where the second equality follows from the smooth case (Corollary~\ref{cor: smooth case}) and Remark~\ref{rk: P=HT for strongly horo}. Since $H \cap T$ contains $H_{\mathrm{red}} \cap T$, and since $H_{\mathrm{red}}$ is obtained by taking kernels of characters of $P_{\mathrm{red}}=\bigcap_{\nu \in \Delta \setminus I} P^\nu$, only the fundamental weights associated to $\nu \in \Delta \setminus I$ can appear when writing $H \cap T$ as an intersection of kernels of characters of $T$. Thus,
\begin{equation} \label{eq: H cap T}
H \cap T = \bigcap_{j \in F} \ker\left(\sum_{\nu \in \Delta \setminus I} l_{\nu,j} p^{s_{\nu,j}} \varpi_\nu \right) \subset T,
\end{equation}
with $F$ a finite set, $s_{\nu,j} \geq 0$, and $l_{\nu,j} \geq 0$, which—when nonzero—must be coprime to $p$.

We will first show that all the $s_{\nu,j}$ that appear in \eqref{eq: H cap T} satisfy $s_{\nu,j} \geq r_\nu$, which will allow us to define the strongly horospherical subgroup
\begin{equation} \label{def: H'}
H' := \bigcap_{j \in F} \ker\left(\sum_{\nu \in \Delta \setminus I} l_{\nu,j} p^{s_{\nu,j}} \varpi_\nu \right) \subset P
\end{equation}
whose normalizer is $P$ (Proposition~\ref{prop characterization strongly horo car p>2}~\eqref{prop characterization strongly horo car p>2 item ii}). Then we will show that $H = H'$.

\smallskip

Let $L := F(H) \simeq H / {}_1 H \subset G^{(1)}$. By the induction hypothesis, since the height of $L$ is one less than that of $H$, the subgroup $L$ is strongly horospherical. 
Additionally, we have
\[
L \cap T^{(1)} = F(H \cap T) = \bigcap_{j \in F} \ker \left( \sum_{\nu \in \Delta \setminus I} l_{\nu,j} p^{s_{\nu,j} - \varepsilon(\nu,j)} \varpi_\nu^{(1)} \right) \subset T^{(1)},
\ \ 
\text{where}
\ \varepsilon(\nu,j) = 
\begin{cases}
1 & \text{if } s_{\nu,j} \geq 1, \\
0 & \text{if } s_{\nu,j} = 0.
\end{cases}
\]
Hence, since $L$ is strongly horospherical, we deduce that
\[
L = \bigcap_{j \in F} \ker \left( \sum_{\nu \in \Delta \setminus I} l_{\nu,j} p^{s_{\nu,j} - \varepsilon(\nu,j)} \varpi_\nu^{(1)} \right) \subset N_{G^{(1)}}(L)=LT^{(1)}.
\]
On the other hand, the character group of $LT^{(1)} = F(P)$ is given by
\[
X^\ast(LT^{(1)}) 
= X^\ast(F(P)) 
= \bigoplus_{\nu \in \Delta \setminus I} 
    \mathbb{Z} \cdot p^{r_\nu - \varepsilon(\nu)} \varpi_\nu^{(1)},
\quad \text{where} \quad
\varepsilon(\nu) =
\begin{cases}
1 & \text{if } r_\nu \geq 1, \\
0 & \text{if } r_\nu = 0.
\end{cases}
\]
Thus, we deduce that for any $\nu \in \Delta \setminus I$, if $l_{\nu,j} \neq 0$, then
\[
s_{\nu,j} - \varepsilon(\nu,j) \geq r_\nu - \varepsilon(\nu),
\]
hence either $s_{\nu,j} \geq r_\nu$, or $(s_{\nu,j}, r_\nu) = (0,1)$.

	It remains to check that, if $l_{\nu,j} \neq 0$ and $r_\nu = 1$, then $s_{\nu,j} \neq 0$. To see this, let $\gamma$ be the sum of all simple roots and denote as $G_\gamma$ the corresponding copy of $\SL_2$, with fundamental weight $\omega$. Since the multiplicity of $\nu$ in $\gamma$ is equal to one, the weight $\varpi_\nu$ restricts to $\omega$ on the maximal torus of $G_\gamma$. Next, consider the intersection $H \cap G_\gamma$, which is a 
strongly horospherical subgroup of $G_\gamma$ (see Section~\ref{sec: SL2 case}). 
Since $r_\nu =1$, there exists an integer $m_\gamma \geq 1$ such that
\begin{equation} \label{eq: restriction to Ggamma}
H \cap G_\gamma = \bigcap_{j \in F} \ker(l_{\nu,j} p^{s_{\nu,j}} \omega) = \ker(m_\gamma p \omega) \subset Q^\nu \cap G_\gamma = {}_1 G_\gamma (B \cap G_\gamma), 
\end{equation}
and
\[
m_\gamma p = \gcd_{j \in F}(l_{\nu,j} p^{s_{\nu,j}}).
\]
In particular, we must have $s_{\nu,j} \geq r_\nu = 1$ for every $j$ with $l_{\nu,j} \neq 0$. We can thus define $H'$ as in \eqref{def: H'}.
Notice that if $p=3$ and the group $G$ has a simple factor of type $G_2$, then the choice of $\gamma$ (in particular, being a \emph{short} root containing $\nu$ in its support) guarantees the last equality in \eqref{eq: restriction to Ggamma}. Indeed, in all other cases $r_\nu=1$ implies necessarily that $Q^\nu = {}_1GP^\nu$, while in this last case the existence of the very special isogeny allows for a second possibility, namely $Q^\nu = KP^\nu$. Here the fact that $\gamma$ is short is crucial: indeed, for a long root $\delta$ containing $\nu$ it its support, we have instead $Q^\nu \cap G_\delta= B \cap G_\delta$.

\smallskip

We claim that $H = H'$. First, by construction, we have
\[
F(H) = F(H') = L.
\]
Thus, to conclude, it suffices to show:
\[
\text{(a)} \quad {}_1 H = {}_1 H', \qquad \text{(b)} \quad  H' \subset H.
\]
Indeed, these two conditions imply that both $H$ and $H'$ are extensions of $L$ by ${}_1 H$, and since one is contained in the other, they must coincide.
\[
\begin{tikzcd}[column sep=large, row sep=small]
0 \arrow[r] 
  & {}_1H \arrow[r] 
  & H \arrow[r] 
  & F(H) \arrow[r] 
  & 0 \\
0 \arrow[r] 
  & {}_1H' \arrow[r] \arrow[u, equal]
  & H' \arrow[r] \arrow[u, hook]
  & F(H') \arrow[r] \arrow[u, equal]
  & 0
\end{tikzcd}
\]

\smallskip

\textbf{(a):} This is equivalent to showing that $\Lie(H) = \Lie(H')$ (see Theorem~\ref{th: DG height one}~\eqref{th: DG height one item i}). 
We know from Proposition~\ref{prop: connected smooth casev2}~\eqref{prop: connected smooth casev2 item i} that both of these Lie subalgebras are $T$-stable and that they contain $\Lie(U)$. 
Moreover, by construction, $H \cap T = H' \cap T$, so
\[
\Lie(H) \cap \Lie(T) = \Lie(H') \cap \Lie(T),
\]
and it remains to check that the intersections with $\Lie(U_-)$ coincide.

On one hand, we have
\begin{equation} \label{eq: inclusion for (a)}
\Lie(H) \cap \Lie(U_-) \subset \Lie(P) \cap \Lie(U_-) = \Lie(H') \cap \Lie(U_-),
\end{equation}
since $P \cap U_- = H' \cap U_-$ by Lemma~\ref{lem heights}, and hence $\Lie(P) \cap \Lie(U_-) = \Lie(H') \cap \Lie(U_-)$.
Assume for contradiction that the inclusion \eqref{eq: inclusion for (a)} is strict; that is, there exists some $\alpha \in \Delta \setminus I$ such that 
\[
\mathfrak{g}_{-\alpha} \cap \Lie(H) = 0 \quad \text{while} \quad \mathfrak{g}_{-\alpha} \cap \Lie(H') \neq 0.
\]
This means that $U_{-\alpha} \cap H = 1$ while $U_{-\alpha} \cap H' \neq 1$. 
By Theorem~\ref{thm: Knop horo subgroups SL2}, and since $p \geq 3$, intersecting with the corresponding copy $G_\alpha \simeq \SL_2$ yields the inclusion
\[
H \cap G_\alpha \subset B \cap G_\alpha,
\]
and thus, multiplying by $T \cap G_\alpha$ gives
\[
P \cap G_\alpha = (HT) \cap G_\alpha \subset B \cap G_\alpha,
\]
which in turn implies $H' \cap U_{-\alpha} = P \cap U_{-\alpha} = 1$, a contradiction.

So the inclusion~\eqref{eq: inclusion for (a)} must be an equality. That is, (a) holds.

	\smallskip

\textbf{(b):} First note that $F(H')=F(H)$ implies that $H'_{\red}=H_{\red}$.
Since $H'$ is strongly horospherical, it is generated by $U$, together with $H' \cap T$ and the groups $H' \cap U_{-\nu}$ for $\nu \in \Delta$ ranging over the simple roots. 
In order to conclude that $H' \subset H$, it is thus enough to verify that
\[
H' \cap T = H \cap T, \qquad 
H' \cap U_{-\nu} = H \cap U_{-\nu}, 
\quad \text{for all } \nu \in \Delta.
\]

The first equality holds by the very definition of $H'$. Let us now examine the second:

\begin{itemize}
  \item If $\nu \in I$, then $U_{-\nu} \subset H_{\mathrm{red}}$, so in particular,
  \[
  U_{-\nu} = U_{-\nu} \cap H = U_{-\nu} \cap H'.
  \]
  
\item If $U_{-\nu} \cap H = 1$, then the same argument as in part~\textbf{(a)} yields $U_{-\nu} \cap H' = 1$.
   
  \item Now assume that $\operatorname{ht}(U_{-\nu} \cap H) = s_\nu \geq 1$, and let $\operatorname{ht}(U_{-\nu} \cap H') = s'_\nu \in \N_{\geq 0}$. Note that, by step~\textbf{(a)}, we must have $s'_\nu \geq 1$. Then
\[
s_\nu - 1 
= \operatorname{ht}(F(H \cap U_{-\nu})) 
= \operatorname{ht}(L \cap U_{-\nu}^{(1)}) 
= \operatorname{ht}(F(H' \cap U_{-\nu})) = s'_\nu - 1 \geq 0,
\]
which immediately gives $s_\nu = s'_\nu$.
\end{itemize}
This concludes the proof of the facts \textbf{(a)} and \textbf{(b)}, and thus also the proof of the proposition.
\end{proof}

\begin{example}
Let us clarify at which step the above proof fails in the unique example of a horospherical subgroup 
$H = \mathcal{B}(\infty) \subset G = \SL_2$ that is not strongly horospherical when $p = 2$. In this case, we have:
\begin{align*}
&H \cap T = \ker(2\varpi_\alpha) \simeq \boldsymbol{\mu}_2 \subset T, \\
&P := HT = F^{-1}(B) \quad \Rightarrow \quad X^\ast(P) = \mathbb{Z} 2\varpi_\alpha, \\
&H' := \ker(2\varpi_\alpha)=F^{-1}(\boldsymbol{\mu}_2 \ltimes U) \subset P.
\end{align*}
Hence, we obtain, as in the proof of Theorem~\ref{th: general case}, a strongly horospherical subgroup $H' \subset \SL_2$. But $H' \not\subset H$ (which is impossible for $p \geq 3$). Indeed,
\[
H \cap U_{-\alpha} = 1 
\quad \text{while} \quad 
H' \cap U_{-\alpha} = \boldsymbol{\alpha}_2.
\]
\end{example}

\subsection*{Acknowledgments}
{We are grateful to the anonymous referee for the careful reading of the former version of this paper and for their helpful comments.}
We also thank K\'evin Langlois, Michel Brion and Friedrich Knop for useful discussions related to this work.

The first author  was supported by the DFG through the research grants Le 3093/5-1 and Le 3093/7-1 (project numbers 530132094; 550535392).

The second author acknowledges the support of the CDP C2EMPI, as well as the French State under the France-2030 programme, the University of Lille, the Initiative of Excellence of the University of Lille, the European Metropolis of Lille for their funding and support of the R-CDP-24-004-C2EMPI project.

\bibliographystyle{alpha} 
	\bibliography{biblio} 	
 
\end{document}